\documentclass[11pt]{amsart}
\usepackage{epsfig}
\usepackage{graphicx, subfigure}
\usepackage{amsmath, amssymb,mathabx,mathrsfs}
\usepackage{color}
\usepackage[toc,page]{appendix}
\usepackage{comment}
\usepackage{hyperref}

\newtheorem{thm}{Theorem}[section]
\newtheorem{cor}[thm]{Corollary}

\newtheorem{prop}[thm]{Proposition}
\theoremstyle{definition}

\theoremstyle{remark}
\newtheorem{rem}[thm]{Remark}

\numberwithin{equation}{section}

\newcommand{\bq}{{\bf q}}
\newcommand{\bM}{{\bf M}}
\newcommand{\bz}{{\bf z}}
\newcommand{\br}{{\bf r}}
\newcommand{\mm}{{\bar{\mu}}}
\newcommand{\cc}{c}
\newcommand{\rx}{r^*_x}
\newcommand{\ry}{r^*_y}
\newcommand{\ryo}{r_{y0}}
\newcommand{\rxo}{r_{x0}}
\newcommand{\rz}{r^*_z}

\renewcommand{\marginpar}[1]{}
\begin{document}

\title[Sun-Jupiter-Hektor-Skamandrios ]
{Hill four-body problem with oblate tertiary: an application to the Sun-Jupiter-Hektor-Skamandrios system }%

\author[J. Burgos-Garc\'ia]{Jaime Burgos-Garc\'ia}
\address{
Autonomous University of Coahuila, C.P. 25020,
Saltillo, Mexico}
\email{jburgos@uadec.edu.mx }

\author[A. Celletti]{Alessandra Celletti}
\address{
Department of Mathematics, University of Roma Tor Vergata, Via
della Ricerca Scientifica 1, 00133 Roma (Italy)}
\email{celletti@mat.uniroma2.it}

\author[C. Gales]{Catalin Gales}
\address{
Department of Mathematics, Al. I. Cuza University, Bd. Carol I 11,
700506 Iasi (Romania)}
\email{cgales@uaic.ro}

\author[M. Gidea]{Marian Gidea$^\dag$}
\address{
Department of Mathematical Sciences, Yeshiva University,
New York, NY 10016 (USA)}
\email{Marian.Gidea@yu.edu}

\author[W.-T. Lam]{Wai-Ting Lam$^\ddag$}
\address{
Department of Mathematical Sciences, Yeshiva University,
New York, NY 10016 (USA)}
\email{WaiTing.Lam@yu.edu}

\begin{abstract}
We consider a restricted four-body problem with a precise
hierarchy between the bodies: two point-mass bigger bodies, a
smaller one with oblate shape, and an infinitesimal body in the neighborhood
of the oblate body. The three heavy bodies are assumed to move in
a plane under their mutual gravity, and the fourth  body moves
under the gravitational influence of the three heavy bodies, but
without affecting them.

We start by finding the triangular central configurations of the three
heavy bodies; since one body is oblate, the triangle is isosceles,
rather than equilateral as in the point mass case. We assume
that the three heavy bodies are in such a central configuration
and we perform a Hill's approximation of the equations of motion
describing the dynamics of the infinitesimal body in a
neighborhood of the oblate body. Through the use of Hill's
variables and a limiting procedure, this approximation amounts to
sending the two other bodies to infinity. Finally, for the Hill
approximation, we find the equilibrium points of the infinitesimal body and
determine their stability. As a motivating example, we consider
the dynamics of the moonlet Skamandrios of Jupiter's  Trojan
asteroid Hektor.
\end{abstract}
\marginpar{MG: abstract edits}

\maketitle

\section{Introduction}\label{section:introduction} \marginpar{MG: introduction edits}

The discovery of binary asteroids has led to considering  dynamical
models formed by four bodies, two of them being typically the Sun
and Jupiter. Among possible four-body models (see also
\cite{HOWELL1986,scheeres1998restricted,gabern2003restricted,scheeres2005restricted,Alvarez2009,Delgado_2012,
burgos2013blue,Burgos_2016,Kepley_James_2017}), a relevant role is
played by the models in which three bodies lie on a triangular
central configuration. Given that asteroids have often a (very)
irregular shape, it is useful to start by investigating the case
in which one body has an oblate shape. Among the different
questions that this model may rise, we concentrate on the
existence of equilibrium points and the corresponding linear
stability analysis. Within such framework, we consider a four-body
simplified model and we concentrate on the specific example given
by the Trojan asteroid 624 Hektor, which is located close to the
Lagrangian point  $L_4$ of the Sun-Jupiter system, and its small
moonlet.

In our model Sun, Jupiter and Hektor form an isosceles triangle
(nearly equilateral) whose shape remains unchanged over time. The
small body represents Hektor's moonlet Skamandrios. Obviously, we
could also replace the moonlet by a spacecraft orbiting Hektor. The
system Sun-Jupiter-Hektor-Skamandrios plays a relevant role for
several reasons. Indeed, Hektor is the largest Jupiter Trojan, it
has one of the most elongated shapes among the bodies of its size
in the Solar system, and it is the only known Trojan to possess a
moonlet (see, e.g., \cite{Dvorak1} for stability regions of
Trojans around the Earth and \cite{LC} for dissipative effects
around the triangular Lagrangian points).


The study of asteroids with satellites presents a special interest
for planetary dynamics, as they provide information about
constraints on the formation and evolution of the Solar system.

Another motivation to study the dynamics of a small body near a
Trojan asteroid comes from astrodynamics, as NASA prepares the
first mission, Lucy, to the Jupiter's Trojans, which is planned to
be launched in October 2021 and visit seven different asteroids: a
Main Belt asteroid and at least five Trojans.


As a model for the Sun-Jupiter-Hektor dynamics, we consider a
system of three bodies of  masses $m_1 \geq m_2 \geq m_3$, which
move in circular orbits under mutual gravity, and form a
triangular central configuration. We refer to these bodies as the
primary, the secondary, and the tertiary, respectively. We assume
that the first two bodies of masses $m_1, m_2$ are spherical and
homogeneous, so they can be treated as point masses, while the
third body of mass $m_3$ is oblate. We describe the gravitational
potential of $m_3$ in terms of spherical harmonics, and we only
retain the most significant ones.  We show  the existence of a
corresponding triangular central configuration, which turns out to
be an isosceles triangle; if the oblateness of the mass $m_3$ is
made to be zero, the central configuration becomes the well-known
equilateral triangle Lagrangian central configuration. We stress
that when $m_3$ is oblate, the central configuration is not the
same as in the non-oblate case, since the overall gravitational
field of $m_3$ is no longer Newtonian;  it is well known that
central configurations depend on the nature of the gravitational
field (see, e.g.,
\cite{Corbera2004,Arredondo_Perez-Chavela,diacu2016central,Martinez2017}).
We note that there exist  papers in the literature (e.g.,
\cite{Asique2016}), which consider systems of three bodies, with
one of the bodies non-spherical, which are assumed to form an
equilateral triangle central configuration. Such assumption, while
it may lead to very good approximations, is not physically
correct.

The moonlet Skamandrios is represented by a fourth body, of
infinitesimal mass, which moves in a vicinity of $m_3$  under the
gravitational influence of   $m_1, m_2, m_3$, but without
affecting their motion. We consider the motion of the
infinitesimal mass taking place in the three-dimensional space; it
is not restricted to the plane of  motion of the three heavy
bodies. This situation is referred to as the spatial circular
restricted four-body problem, and can be described by an
autonomous Hamiltonian system of $3$-degrees of freedom.

We `zoom-in'  onto the dynamics  in a small neighborhood of $m_3$
by  performing a Hill's approximation of the  restricted four-body
problem. This is done by a rescaling of the coordinates  in terms
of $m_3^{1/3}$,  writing  the associated  Hamiltonian in the
rescaled coordinates as a power series in $m_3^{1/3}$, and
neglecting all the terms of order $O(m_3^{1/3})$  in  the expansion,
since such terms are  small when $m_3$ is small.
This yields an approximation of the motion of
the massless particle in an $O(m_3^{1/3})$-neighborhood  of $m_3$,
while $m_1$ and $m_2$ are `sent to infinity' through the
rescaling. \marginpar{MG:description of the Hill approximation changed}
This model is an extension of the
classical lunar Hill problem \cite{Hill}. Since the tertiary
is assumed to be oblate, and the corresponding central
configuration formed by the three heavy bodies  is not an
equilateral triangle anymore, this model also extends Hill's
approximation of the restricted four-body problem  developed in
\cite{Burgos_Gidea}.

The Hill approximation is more advantageous to utilize for this
system than the restricted four-body problem, since it allows for
an analytical treatment, and yields more accurate numerical
implementations when realistic parameters are used.  The main
numerical difficulty in the restricted four-body problem is the
large differences of scales among the relevant parameters, i.e.
the mass of Hektor is much smaller than the masses of the other
two heavy bodies. The rescaling of the coordinates involved in the
Hill  approximation reduces the difference of scales of the
parameters to more manageable quantities; more precisely, in
normalized units  the oblateness effect in the restricted
four-body problem is of the order $O(10^{-15})$, while in the Hill
approximation is of the order $O(10^{-7})$ (see Section
\ref{sec:Hill_system} for details).

Once we have established the model for the  Hill four-body problem
with oblate tertiary, we study the equilibrium points and their
linear stability. We find that there are $2$ pairs of symmetric
equilibrium points on each of the $x$-, $y$-, and $z$-coordinate
axes, respectively. The equilibrium points on the $x$- and
$y$-coordinate axes  are just a continuation of the corresponding
ones for the  Hill four-body problem with non-oblate tertiary
\cite{Burgos_Gidea}. The equilibrium points on  the $z$-coordinate
axis constitute a  new feature of the model. In the case of
Hektor, these equilibrium points turn out to be outside of the
body of the asteroid but very close to the surface, so they are of
potential interest for low altitude orbit space missions, such as
the one of  NASA/JPL's Dawn mission around Vesta
(\cite{delsate2011analytical}).

This work is organized as follows. In Section \ref{sec:model} we
describe in full details the restricted four-body model in which
the tertiary is oblate; in particular, we describe the isosceles triangle central configuration
of three bodies
in which two bodies are point masses and the third is oblate.
Hill's approximation is introduced in
Section \ref{sec:Hill}. The determination of the equilibria and
their stability is given in Section \ref{sec:linear}.

\section{Restricted four-body problem with oblate tertiary}\label{sec:model}
In this section we develop a model for
a restricted four-body problem, which consists of two bigger
bodies (e.g., the Sun and Jupiter), a smaller body -- called
tertiary -- with oblate shape (e.g. an asteroid),
and an infinitesimal mass (e.g.,
moonlet) around the tertiary.

As mentioned in Section \ref{section:introduction}, we consider the three masses
$m_1 \geq m_2 \geq m_3$ as moving under the mutual gravitational attraction;
the bodies with masses $m_1$ and $m_2$ are considered as point masses, while $m_3$ is the oblate body.
We normalize the units of mass so that $m_1+m_2+m_3=1$.

We assume that the bodies with masses $m_1$, $m_2$, $m_3$ move on a triangular central configuration,
which will be determined in Section~\ref{sec:central_config}, once the gravitational field of the
oblate body has been discussed in Section~\ref{sec:nonspheriacal}. We will concentrate on the specific example
given by the asteroid Hektor and its moon Skamandrios, where Hektor moves on a central configuration
with Jupiter and the Sun. Orbital and physical values are given in Section~\ref{section:data}. The positions
of the three man bodies in the triangular central configuration is computed in Section~\ref{sec:location_central_config},
while the equations of motion of the moonlet -- with infinitesimal mass moving
in the vicinity of $m_3$ -- are given in Section~\ref{sec:4BP}.

\subsection{Data on the Sun-Jupiter-Hektor-Skamandrios system}\label{section:data}

The models which we will develop below will be applied to the case of the Sun-Jupiter-Hektor-Skamandrios system.
We extract the data for this system from \cite{JPL,Marchis,DESCAMPS2015}.

Hektor is approximately located at  the Lagrangian point $L_4$ of the Sun-Jupiter system.
According to \cite{DESCAMPS2015}, Hektor is approximately $416 \times 131 \times 120$ km in size, and its shape can be approximated by a dumb-bell figure;
the equivalent radius (i.e., the radius of a sphere with the same
volume as the asteroid) is $R_H=92$ km\footnote{Note that \cite{DESCAMPS2015} claims that there are some typos
in the values reported in \cite{Marchis}.}.\marginpar{MG:description changed}

Hektor spins very fast, with a rotation period of approximately $6.92$ hours (see the JPL Solar System Dynamics archive
\cite{JPL}).

The moonlet Skamandrios orbits around Hektor at a distance of approximately $957.5$ km, with an orbital period of $2.965079$ days;  see \cite{DESCAMPS2015}.
Its orbit is highly inclined, at approximately $50.1^\circ$ with respect to the orbit of Hektor, which justifies choosing as  a model the spatial restricted  four-body problem rather than the planar one; see \cite{Marchis}.

We also note that the inclination of Hektor is approximately $18.17^\circ$ (see \cite{JPL}).
Although a more refined model should include a non-zero inclination, we will consider  that Sun-Jupiter-Hektor
move in the same plane, an assumption that is needed in order for the three bodies to form a central configuration.
We will further assume that the axis of rotation of Hektor is perpendicular to the plane of motion.

For the masses of Sun, Jupiter and Hektor we use the values of $m_1= 1.989\times10^{30}$ kg, $m_2=1.898\times10^{27}$, and $m_3=7.91\times10^{18}$ kg, respectively. For the average distance Sun-Jupiter we use the value  $778.5\times 10^6$ km.

In Figure \ref{fig:Hektor_forces} we provide a comparison between the strength of the different forces acting on
the moonlet: the Newtonian gravitational attraction of Hektor, Sun, Jupiter, and the effect of the non-spherical shape of the asteroid, limited to  the  the so-called $J_2$ coefficient, which will be introduced in Section \ref{sec:nonspheriacal}.

\begin{figure}\label{fig:Hektor_forces}
\includegraphics[width=0.8\textwidth]{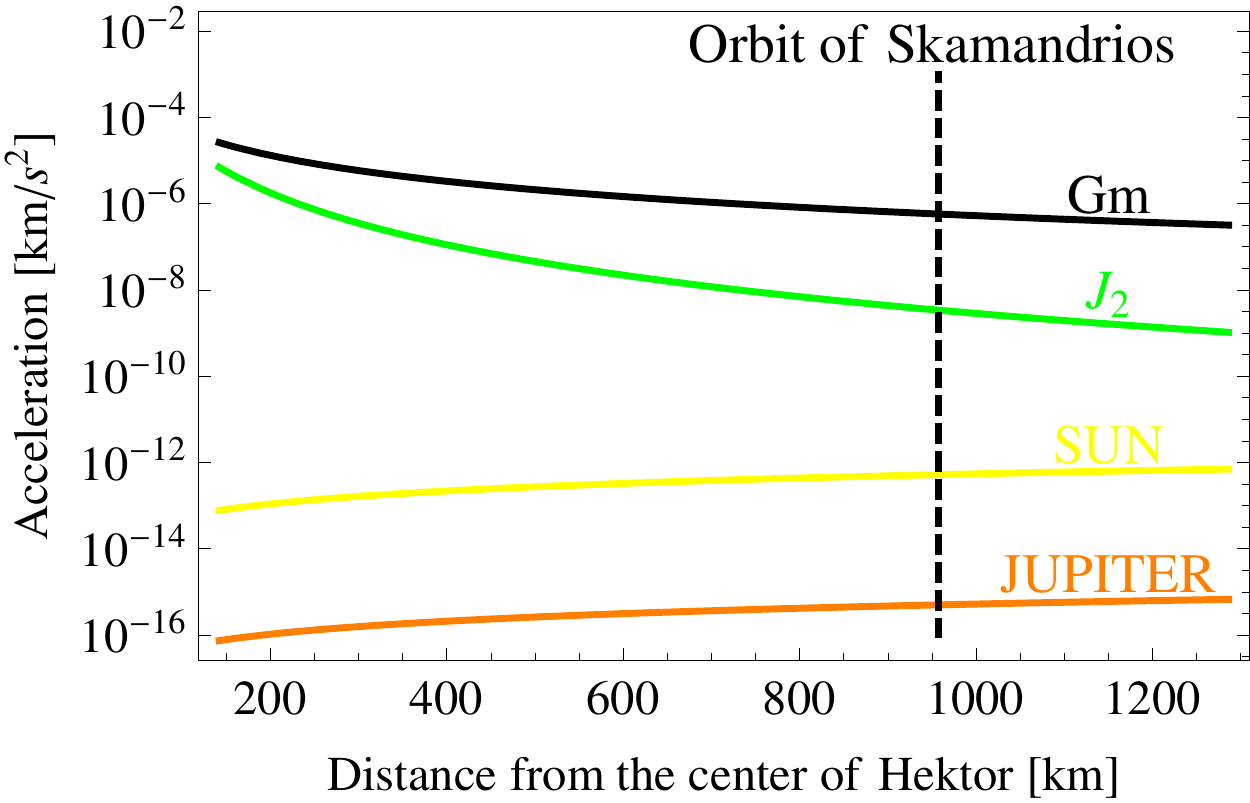}
\caption{Order of magnitude of the different perturbations acting
on the moonlet as a function of its distance from Hektor. The
terms Gm, Sun and Jupiter denote, respectively, the monopole terms
of the gravitational influence of Hektor, the attraction of the
Sun and that of Jupiter. $J_2$ represents the  perturbation due to the
non-spherical shape of Hektor. The actual distance of the moonlet
is indicated by a vertical line.}
\end{figure}

\subsection{The gravitational field of a non-spherical body}\label{sec:nonspheriacal}
We first consider that the tertiary body, representing Hektor, has a general (non-spherical) shape. The gravitational potential, relative to a  reference frame centered at the barycenter of the tertiary and rotating with the body, is given in spherical coordinates $(r,\phi,\lambda)$ by (see, e.g.,  \cite{celletti2018dynamics}):

\begin{equation*}\label{potential}
V(r,\phi,\lambda)={{\mathcal{G} m_H}\over r}\ \sum_{n=0}^\infty \Bigl({R_H\over r}\Bigr)^n\ \sum_{m=0}^n P_{nm}(\sin\phi)\ (C_{nm}\cos m\lambda+
S_{nm}\sin m\lambda)\ ,
\end{equation*}
where $\mathcal{G}$ is the gravitational constant, $m_H$ is the mass of Hektor,  $R_H$ is its average radius,
 $P_{nm}$ are the Legendre polynomials defined as
\begin{equation*}\begin{split}
P_n(x)&= {1\over {2^n n!}}\ {{d^n}\over {dx^n}}(x^2-1)^n\\
P_{nm}(x)&=(1-x^2)^{m\over 2}\ {{d^m}\over {dx^m}}P_n(x)\ ,
\end{split}
\end{equation*}
and  $C_{nm}$ and $S_{nm}$ are the spherical harmonics coefficients.

In the case of an ellipsoid of semi-axes $a\geq b\geq c$, we have the following explicit formulas (\cite{Boyce1997}):
\begin{eqnarray*}
S_{n,m}&=&0,\\
C_{2p+1,2q}&=&0,\\
C_{2p,2q+1}&=&0,\\
C_{2p,2q}&=&\displaystyle{3\over {R_H^{2p}}} {{p!(2p-2q)!}\over {2^{2q}(2p+3)(2p+1)!}}
(2-\delta_{0q})\\
&&\displaystyle\sum_{i=0}^{\lfloor{{p-q}\over 2}\rfloor} {{(a^2-b^2)^{q+2i}[c^2-{1\over 2}(a^2+b^2)]^{p-q-2i}}\over
{16^i(p-q-2i)!(p+q)!i!}}. \end{eqnarray*}

In particular, $C_{20}$  and $C_{22}$ turn out to be given by simple expressions
\begin{eqnarray}
C_{20}&=&\frac{c^2-\frac{a^2}{2}-\frac{b^2}{2}}{5R_H^2},\\
C_{22}&=&\frac{\frac{a^2}{4}-\frac{b^2}{4}}{5R_H^2}.
\end{eqnarray}

For the Sun-Jupiter-Hektor data,  we take $a=208$ km, $b=65.5$ km, $c=60$ km, $R_H=92$ km, following  \cite{DESCAMPS2015} (see Section \ref{section:data}),
we calculate the following coefficients:
\[\begin{tabular}{llll}
  $C_{20}=-0.476775$; &  $C_{22}=0.230232 $; & & \\
  $C_{40}=0.714275$; &  $C_{42}=-0.078406 $; & $C_{44}=0.009465  $; &\\
  $C_{60}=-1.54769$; &  $C_{62}=0.076832  $; &$C_{64}=-0.002507   $; & $C_{66}=0.000201  $.\\
\end{tabular}\]

Notice that each term $C_{2p,2q}$ is multiplied in \eqref{potential} by the factor $R_H^{2p}/r^{2p+1}$.
For $r$ equal to the average distance from the moonlet to the asteroid, we have $R_H/r\approx 0.096$.\marginpar{MG:explanation on $r$ in $R_h/r$} Therefore, in the following we will ignore the effect of the coefficients $C_{2p,2q}$, with $p\geq 2$,
which are at least of $O(R_H^4/r^5)$.

Note that the value of $C_{20}$ computed above is significantly bigger in absolute value than the one reported
in \cite{Marchis}, which equals $-0.15$. The reason is that we use different estimates for the size of Hektor,
following \cite{DESCAMPS2015} (see Section \ref{section:data}).

If we consider a frame centered at the barycenter of the tertiary, and which rotates with the angular
velocity of the tertiary about the primary, the time dependent gravitational potential takes of the form
\begin{equation*}\label{potential_rot2}
V(r,\phi,\lambda)={{\mathcal{G} m_H}\over r}\ \sum_{n=0}^{2} \Bigl({R_H\over r}\Bigr)^n\ \sum_{m=0}^n P_{nm}(\sin\phi)\ C_{nm}\cos (m(\lambda+\Theta t)),
\end{equation*}
where $\Theta$ represents the frequency of the spin of Hektor.

For $n=2$, $m=0$ the corresponding term $C_{nm}\cos (m(\lambda+\Theta t))$ in the summation \eqref{potential_rot2}
is equal to $C_{20}$ and is independent of time; for $n=2$, $m=2$ the corresponding term $C_{nm}\cos (m(\lambda+\Theta t))$
is equal to $C_{22}\cos (2(\lambda+\Theta t))$, so is time-dependent. We do not consider the other terms
in the sum \eqref{potential_rot2}.

Since the ratio of the rotation period of Hektor to the orbital period of the moonlet is relatively small,
approximately $0.09740991$, in this paper we will only consider the average effect of
$C_{22}\cos (2(\lambda+\Theta t))$ on the moonlet, which is zero.

In conclusion, in the model below we will only consider the effect of $C_{20}:=-J_2<0$, which amounts to approximating Hektor
as an oblate body (i.e., an ellipsoid of revolution obtained by rotating an ellipse about its minor axis);
the dimensionless quantity $J_2$ is referred to as the {\sl zonal harmonic}  in the gravitational potential.
The term corresponding to $C_{20}$ is the larger one, followed by that corresponding to $C_{22}$;
however, since the $C_{22}$ term introduces a time dependence, thus further complicating the model,
we start by disregarding it and plan to study its effect in a future work.

\subsection{Central configurations for the three-body problem with one oblate body}
\label{sec:central_config}

We now consider only the  three heavy bodies, of masses $m_1 \geq m_2 \geq m_3$, with the body of mass $m_3$ being oblate,  in which case we  only take into account the term corresponding to $C_{20}=-J_2$ in \eqref{potential}.  We write the
approximation of the gravitational potential of the tertiary \eqref{potential_rot2} in both Cartesian and
spherical coordinates (in the frame of the tertiary and rotating with the body):
\begin{equation}\begin{split}
\label{eqn:C20} V(x,y,z)&=\frac{m_3}{r}- \frac{m_3}{r} \left(\frac{R_3}{r}\right)^2 \left(\frac{J_2}{2}\right)
\left (3 \left (\frac{z}{r}\right)^2 -1\right)\\
&=\frac{m_3}{r}+ \frac{m_3}{r} \left(\frac{R_3}{r}\right)^2 \left(\frac{C_{20}}{2}\right)\left (3 \sin\phi^2 -1\right),
\end{split}
\end{equation}
where $m_3$ is the normalized mass of Hektor (the sum of the three masses is the unit of mass),
$R_3:=R_H$ is the average radius of Hektor in normalized units (the  distance between Sun and Hektor is the unit of distance), the gravitational constant is normalized to $1$,  and $\sin\phi=z/r$.

We want to find the triangular central configurations formed by $m_1$, $m_2$, $m_3$; we will follow the approach in \cite{Arredondo_Perez-Chavela}. Since for a central configuration the three bodies lie in the same plane, in the gravitational field \eqref{eqn:C20} of $m_3$ we set $\phi=0$, obtaining
\begin{equation}
\label{eqn:C20plane} V(q)=\frac{m_3}{r}+ \frac{Cm_3}{r^3},
\end{equation}
where $q=(x,y)$ is the position vector of an arbitrary point in the plane, $r=\|q\|$ is the distance from $m_3$, and we denote \begin{equation}
\label{eqn:C_const} C=R_3^2J_2/2>0.\end{equation}

Let $q_i$ be the position vector of the mass $m_i$, for $i=1,2,3$,  in an inertial frame centered at the
barycenter of the three bodies.

The equations of motion of the three bodies are
 \begin{equation}\label{eqn:3bp}
    \begin{split}
    m_1\ddot{q}_1&={m_1m_2(q_2-q_1)}\frac{1}{\|q_2-q_1\|^3}+{m_1m_3(q_3-q_1)}\left[\frac{1}{\|q_3-q_1\|^3}+\frac{3C}{\|q_3-q_1\|^5}\right],\\
    m_2\ddot{q}_2&={m_2m_1(q_1-q_2)}\frac{1}{\|q_1-q_2\|^3}+{m_2m_3(q_3-q_2)}\left[\frac{1}{\|q_3-q_2\|^3}+\frac{3C}{\|q_3-q_2\|^5}\right],\\
    m_3\ddot{q}_3&={m_3m_1(q_1-q_3)}\frac{1}{\|q_1-q_3\|^3}+{m_3m_2(q_2-q_3)}\frac{1}{\|q_2-q_3\|^3},
    \end{split}
 \end{equation}
where the last terms in the first two equations are due to \eqref{eqn:C20plane}, and the gravitational constant is normalized
to $\mathcal{G}=1$.
Denote $r_{ij}=\|q_i-q_j\|$, for $i\neq j$,   $\bq=(q_1,q_2,q_3)$, and $\bM=\textrm{diag}(m_1, m_1,m_2 ,m_2, m_3,  m_3)$  the $6\times 6$ matrix with $2$ copies of each mass along the diagonal. Then \eqref{eqn:3bp} can be written as
\begin{equation}\label{eqn:v3bp}
    {\bM}\ddot{\bq}=\nabla U(\bq),
\end{equation}
where
\begin{equation}\label{eqn:U}
    U(\bq)={m_1m_2}\frac{1}{r_{12}}+{m_1m_3}\left(\frac{1}{r_{13}}+\frac{C}{r_{13}^3}\right)+
    {m_2m_3}\left(\frac{1}{r_{23}}+\frac{C}{r_{23}^3}\right)
\end{equation}
is the potential for the three body problem with oblate $m_3$.

Let us assume that the center of mass is fixed at the origin, i.e.,
\begin{equation}\label{eqn:cm}
   \bM\bq= \sum_{i=1}^3 m_iq_i=0.
\end{equation}

We are  interested in \emph{relative equilibrium} solutions for the motion of the three bodies,
which are characterized by the fact they become equilibrium points in a uniformly rotation frame.

Denote by $R(\theta)$ the $6\times 6$  block diagonal matrix consisting of  $3$ diagonal blocks the form
\[ \left(
     \begin{array}{rr}
       \cos(\theta) & -\sin(\theta) \\
       \sin(\theta) & \cos(\theta) \\
     \end{array}
   \right)\in SO(2).
\]
Substituting $\bq(t)=R(\omega t)\bz(t)$ for some $\omega\in\mathbb{R}$ in \eqref{eqn:v3bp}, where $\bz=(z_1,z_2,z_3)\in\mathbb{R}^6$, we obtain
\[\label{eqn:E}
\bM\left(\ddot \bz+2\omega J\dot\bz-\omega^2\bz\right)=\nabla U,
\]
where $J$ is the block diagonal matrix consisting of  $3$ diagonal blocks the form
\begin{equation}\label{eqn:eq_x} \left(
     \begin{array}{rr}
       0 & -1 \\
       1 & 0 \\
     \end{array}
   \right).
\end{equation}
The condition for an equilibrium point of \eqref{eqn:E} yields the algebraic equation
\begin{equation}\label{eqn:CC1}
    \nabla U(\bz)+\omega^2 \bM \bz=0.
\end{equation}

A solution $\bz$  of the three-body problem satisfying \eqref{eqn:CC1} is referred to as a  \emph{central configuration}. This is equivalent to  $\ddot z_i=-\omega^2 z_i$, for $i=1,2,3$, meaning that the accelerations of the masses are proportional to the corresponding position vectors, and all accelerations are pointing towards the center of
mass. Thus, the solution $\bq(t)$ is a relative equilibrium solution if and only if $\bq(t)=R(\omega t)\bz(t)$ with $\bz(t)$ being a central configuration solution, and the rotation $R(\omega t)$ being a circular solution of the Kepler problem.

Let $I(\bz)=\bz ^T M \bz=\sum_{i} m_i \|z_i\|^2$ be the moment of inertia. It is easy to see that this is a conserved quantity for the motion, that is, $I(\bz(t))=\bar I$ for
some $\bar I$ at all $t$.  Using Lagrange's second identity (see, e.g., \cite{Gidea-Niculescu}), and that $\bM\bz=0$,
normalizing the masses so that $\sum_{i=1}^3 m_i=1$, the moment of inertia can be written as:
\begin{equation}\label{eqn:I}I(\bz)=\sum_{1\leq i<j\leq 3} m_{ij}  \|z_i-z_j\|^2:=\sum_{1\leq i<j\leq 3} m_{ij}  r_{ij}^2.\end{equation}
Thus, central configurations correspond to critical points of the potential $U$ on the sphere $\bz^TM\bz=1$, which can be obtained by solving the Lagrange multiplier problem
\begin{equation}\label{eqn:CC_LM_x}
\nabla  f(\bz)=0,\qquad I(\bz)-\bar{I}=0,
\end{equation}
where $f(\bz)=U(\bz)+\frac{1}{2}\omega ^2(I(\bz)-\bar{I})$.
In the above, we used the fact that $\nabla I(\bz)=2\bM\bz$.

We solve this problem in the variables $r_{ij}=\|z_i-z_j\|$ for $1\leq i<j\leq 3$, since both $U$ and $I$ can be written in terms of these variables. This reduces the dimension of the system  \eqref{eqn:CC_LM_x} from $7$ equations to $4$ equations.  Denote $\br=(r_{12}, r_{13},r_{23})$, and let $\tilde{f}(\br)$ be the function $f$ expressed in the variable $\br$, that is $\tilde {f}(\br(\bz))=f(\bz)$. By the chain rule, $\nabla_r \tilde {f} \cdot \left (\frac{\partial\br}{\partial\bz} \right )=\nabla_\bz f(\bz)$.  It is easy to see that the rank of the matrix $\left (\frac{\partial\br}{\partial\bz} \right )$ is maximal provided that $z_1,z_2,z_3$ are not collinear (for details, see \cite{Corbera2004,Arredondo_Perez-Chavela}). As we are looking for triangular central configurations, this condition is satisfied.
Thus, $\nabla_r \tilde {f}(\br)=0$ if and only if $\nabla_\bz f(\bz)=0$. In other words, we can now solve the system \eqref{eqn:CC_LM_x} in the variable $\br$. We obtain
\begin{equation}\label{eqn:CC3}
\begin{split}
-\frac{1}{r_{12}^2}+\omega^2r_{12}=0,\\
-\frac{1}{r_{13}^2}-\frac{3C}{r_{13}^4}+\omega^2r_{13}=0,\\
-\frac{1}{r_{23}^2}-\frac{3C}{r_{23}^4}+\omega^2r_{23}=0,\\
m_1m_2r_{12}^2+m_1m_3r_{13}^2+m_2m_3r_{23}^2=\bar{I}.
\end{split}
\end{equation}
Note that the function $h(r)=-\frac{1}{r^3}-\frac{3C}{r^5}$ has positive derivative
$h'(r)=\frac{3}{r^4}+\frac{15C}{r^6}$ with $C,r>0$, hence it is injective.
Thus, the second and third equation in \eqref{eqn:CC3} yield $r_{13}=r_{23}:=u$.
Solving for $\omega$ in the first and second equation we obtain:
\begin{equation}\label{eqn:omega}
\omega=\sqrt{\frac{1}{r_{12}^3}}=\sqrt{\frac{1}{r_{13}^3}+\frac{3C}{r_{13}^5}}.
\end{equation}
Solving for $r_{12}$ yields:
\begin{equation}\label{eqn:uv}r_{12}:=v=\left (\frac{u^5}{u^2+3C} \right)^{1/3}.\end{equation}
Notice that $C>0$ implies $0<v<u$.

The condition $I(\br)=\bar{I}$ yields
\[m_1m_2 \left (\frac{u^5}{u^2+3C} \right)^{2/3}+(m_1m_3+m_2m_3)u^2=\bar{I},\] which is equivalent to
\[\frac{m_1^3m_2^3u^{10}}{\left(u^2+3C\right)^2}  =\left(\bar{I}-(m_1m_3+m_2m_3)u^2 \right )^3.\]
To simplify the notation, let $u^2=z$, $m_{1}^3m_{2}^3=a$, $m_1m_3+m_2m_3=c$, $3C=b$, obtaining
\[\frac{az^5}{(z+b)^2}= (\bar{I}-cz)^3.\]
The function $k(z)=\displaystyle \frac{az^5}{(z+b)^2}$ has derivative  \[k'(z)= \frac{3az^6+8abz^5+5ab^2z^4}{(z+b)^4}>0\]
and the function $l(z)=(\bar{I}-cz)^3$ has derivative \[l'(z)=-3c(\bar{I}-cz)^2<0.\]
Since \[k(0)=0 \textrm{ and } \lim_{z\to+\infty} k(z)=+\infty,\] and \[l(0)=\bar{I}^3>0\textrm{ and }\lim _{z\to +\infty}l(z)=-\infty,\] the equation $k(z)=l(z)$ has a unique solution with $z>0$. We conclude that for each fixed $\bar{I}$, there is a unique solution of \eqref{eqn:CC3}.
Thus, we have proved the following result.

\begin{prop}\label{prop:CC} In the three-body problem with one oblate primary, for every fixed value $\bar{I}$ of the moment of inertia there exists a unique central configuration, which is an isosceles triangle.
\end{prop}

We note that, while \cite{Arredondo_Perez-Chavela} studies central configurations of three oblate bodies (as well as of three bodies under  Schwarzschild  metric),  the isosceles central configuration found above is not explicitly shown there (see Theorem 4 in \cite{Arredondo_Perez-Chavela}).

To put this in quantitative perspective, when we use the data from Section \ref{section:data} in \eqref{eqn:C_const}, we obtain $C=3.329215\times 10^{-15}$.
If we set $u=r_{13}=r_{23}=1$, from \eqref{eqn:uv} we obtain $v=r_{12}=0.9999999999999967=1.0-3.3\times 10^{-15}$. In terms of the  Sun-Jupiter distance $r_{12}=778.5\times 10^6$ km, the distance $r_{13}=r_{23}$  differs from the corresponding distance in the equilateral central configuration by  $2.7\times 10^{-6}$ km.
Practically, this isosceles triangle central configuration is almost an equilateral triangle.

\subsection{Location of the bodies in the triangular central configuration}
\label{sec:location_central_config}
We now compute the expression of the location of the three bodies in the triangular central configuration,
relative to a synodic  frame that rotates together with the bodies,
with the center of mass fixed at the origin, and the location of $m_1$
on the negative $x$-semi-axis. We assume that the masses lie in the $z=0$ plane.
Instead of fixing the value $\bar{I}$ of the moment of inertia, we fix $u=r_{13}=r_{23}=1$
and have $v=r_{12}<1$ given  by \eqref{eqn:uv}. Then, we obtain the following result.

\begin{figure}
\includegraphics[width=0.65\textwidth]{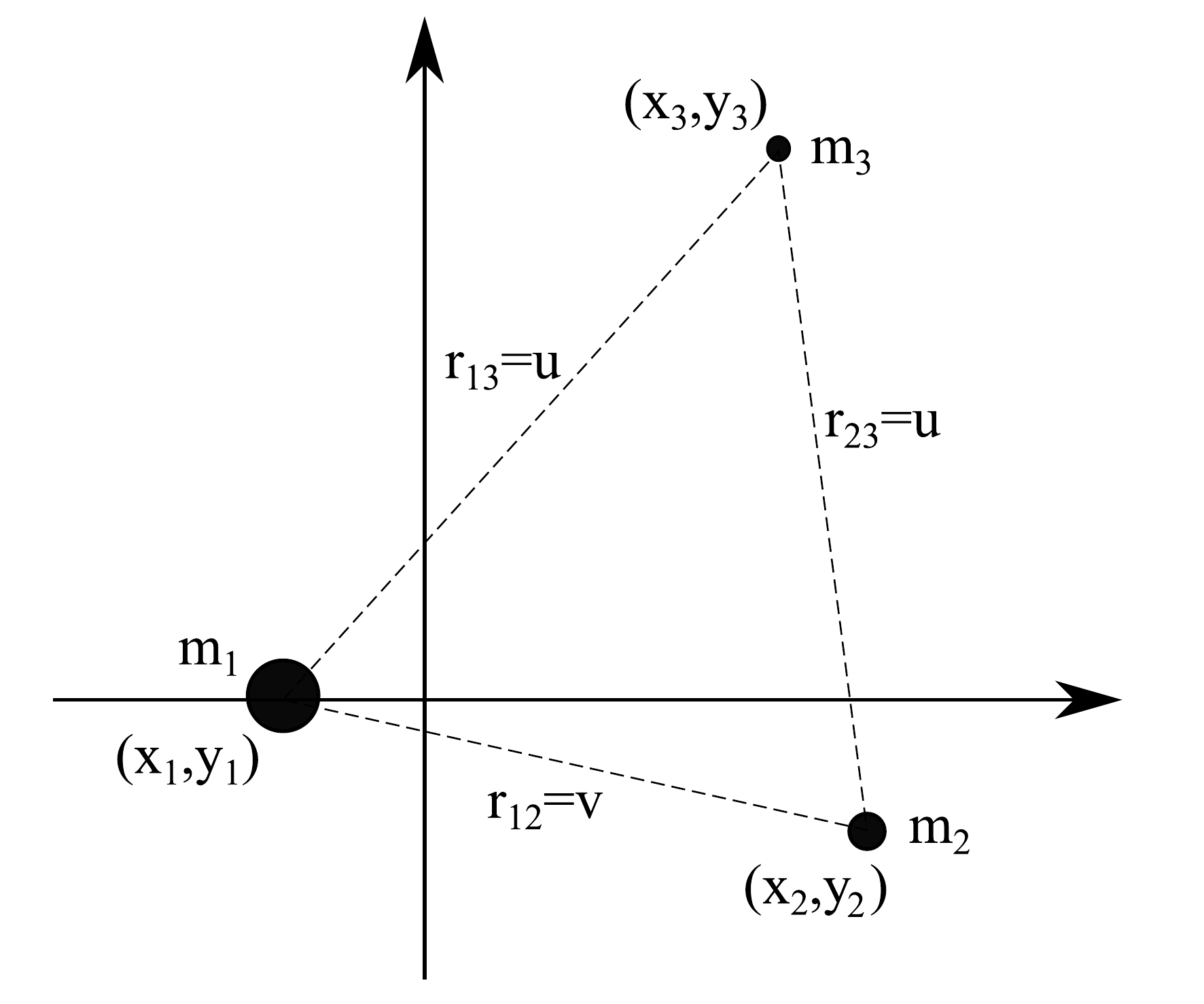}
\caption{Triangular central configuration.}
\label{central_config}
\end{figure}

\begin{prop}
In the synodic reference frame, the coordinates of the three bodies in the triangular central configuration,
satisfying the constraints
\begin{eqnarray}
\label{eqn:1} (x_2-x_1)^2+(y_2-y_1)^2 &=& v^2,\\
\label{eqn:2} (x_3-x_1)^2+(y_3-y_1)^2 &=& 1,\\
\label{eqn:3} (x_3-x_2)^2+(y_3-y_2)^2 &=& 1,\\
\label{eqn:4} m_1x_1+m_2x_2+m_3x_3 &=& 0,\\
\label{eqn:5} m_1y_1+m_2y_2+m_3y_3 &=& 0,\\
\label{eqn:6} m_1+m_2+m_3&=&1,\\
\label{eqn:7} y_1&=&0
\end{eqnarray}
are given by
\begin{equation}
\label{eqn:xy_CC}
\begin{split}
x_1=&-\sqrt{v^2m_2^2+v^2m_2m_3+m_3^2},\\
y_1=&0,\\
x_2=&\frac{2v^2m_2+v^2m_3-2v^2m_2^2-2v^2m_2m_3-2m_3^2}{2\sqrt{v^2m_2^2+v^2m_2m_3+m_3^2}},\\
y_2=&-\frac{v\sqrt{4-v^2} m_3}{2\sqrt{v^2m_2^2+v^2m_2m_3+m_3^2}},\\
x_3=& \frac{v^2m_2+2m_3-2v^2m_2^2-2v^2m_2m_3-2m_3^2}{2\sqrt{v^2m_2^2+v^2m_2m_3+m_3^2}},\\
y_3=&\frac{v\sqrt{4-v^2} m_2}{2\sqrt{v^2m_2^2+v^2m_2m_3+m_3^2}}.
\end{split}
\end{equation}
\end{prop}

\begin{proof}
Denote by $A=x_1-x_2$ and $B=x_1-x_3$, so $x_2=x_1-A$, $x_3=x_1-B$, and $x_3-x_2=A-B$.
Substituting these in  \eqref{eqn:4} we obtain $(m_1+m_2+m_3)x_1=m_2A+m_3B$. From \eqref{eqn:6} it follows $x_1=m_2A+m_3B=m_3(\mm A+B)$, where we denoted $\mm:=m_2/m_3$.
From \eqref{eqn:7} and \eqref{eqn:5} we have $y_3=-(m_2/m_3)y_2=-\mm y_2$, and $y_3-y_2=-(1+\mm)y_2$.
From \eqref{eqn:1}, we can solve for $y_2$ in terms of $A$ (see \eqref{eqn:11} below). From now on, the objective is to solve for
$A$ and $B$, which in turn will yield $x_1,x_2,x_3,y_2, y_3$.

Equations \eqref{eqn:1}, \eqref{eqn:2}, \eqref{eqn:3} become:
\begin{eqnarray}
\label{eqn:11} A^2+y_2^2 &=& v^2,\\
\label{eqn:12} B^2+\mm^2y_2^2 &=& 1,\\
\label{eqn:13} (B-A)^2+(1+\mm)^2y_2^2 &=& 1.
\end{eqnarray}
Adding \eqref{eqn:11} and  \eqref{eqn:12}, and subtracting \eqref{eqn:13} yields
\begin{equation}
\label{eqn:14}
2BA-2\mm y_2^2=v^2.
\end{equation}
From \eqref{eqn:11},  $y_2^2=v^2-A^2$, so \eqref{eqn:14} yields
\begin{equation}
\label{eqn:15}
B=\frac{v^2+2\mm(v^2-A^2)}{2A}.
\end{equation}
Substituting in \eqref{eqn:12} and solving for $A$ yields
\begin{equation}
\label{eqn:15}
A=\pm\frac{v^2(2\mm+1)}{2\sqrt{v^2\mm^2+v^2\mm+1}}.
\end{equation}
Substituting  $y_2^2=v^2-A^2$ in \eqref{eqn:14} and solving for $B$ yields
\begin{equation}
\label{eqn:16}
B= \frac{v^2}{4A}\frac{(2\mm+1)(v^2\mm+2)}{v^2\mm^2+v^2\mm+1}=\pm\frac{v^2\mm+2}{2\sqrt{v^2\mm^2+v^2\mm+1}}.
\end{equation}
with $\textrm{sign}(A)=\textrm{sign}(B)$.

Substituting $A$, $B$ and $\mm$ in  $x_1=m_2A+m_3B$,  and choosing the negative sign for $x_1$ to agree
with our initial choice that $x_1<0$,  after simplification, we first obtain the value of $x_1$ below.
Then, substituting in $x_2=x_1-A$, $x_3=x_1-B$, $y_2^2=v^2-A^2$, $y_3=-(m_2/m_3)y_2$,
we compute $x_2$, $x_3$, $y_2$, $y_3$, obtaining \eqref{eqn:xy_CC}.

For future reference, we note that if we let $m_3\to 0$ in \eqref{eqn:xy_CC}, we obtain
\begin{equation}
\label{eqn:xy_CC_m3_zero}
\begin{split}
x_1=&-vm_2,\\
y_1=&0,\\
x_2=&v(1-m_2),\\
y_2=&0,\\
x_3=&\frac{v}{2}(1-2m_2),\\
y_3=&\frac{\sqrt{4-v^2}}{2}.
\end{split}
\end{equation}
\end{proof}

\begin{rem} In the case when the oblateness coefficient  $J_2$ of $m_3$ is made equal to zero, then $v=1$, and in \eqref{eqn:xy_CC} we obtain the Lagrangian equilateral triangle central configuration,
with the position given by the following equivalent formulas (see, e.g.,  \cite{Baltagiannis2013}):
\begin{equation} \label{eqn:Baltagiannis}
\begin{split}
x_1 & =\frac{-\lvert K \rvert \sqrt{m_2^2+m_2 m_3+m_3^2}}{K},\\
y_1 &=0,\\
\vspace{3mm}
x_2 & =\frac{\lvert K \rvert [(m_2-m_3)m_3+m_1(2m_2+m_3)]}{2K\sqrt{m_2^2+m_2m_3+m_3^2}},\\
y_2 &=\frac{-\sqrt{3}m_3}{2m_2^{\frac{3}{2}}}\sqrt{\frac{m_2^3}{m_2^2+m_2m_3+m_3^2}},\\
\vspace{3mm}
x_3 & =\frac{\lvert K \rvert }{2\sqrt{m_2^2+m_2m_3+m_3^2}},\\
y_3 &=\frac{\sqrt{3}}{2m_2^{\frac{1}{2}}}\sqrt{\frac{m_2^3}{m_2^2+m_2m_3+m_3^2}},
\end{split}
\end{equation}
where $K=m_2(m_3-m_2)+m_1(m_2+2m_3)$.

Notice that the equations \eqref{eqn:Baltagiannis} are expressed in terms of $m_1,m_2,m_3$, while \eqref{eqn:xy_CC} are expressed in terms of $m_2, m_3$; we obtain corresponding expressions that are equivalent when we substitute $m_1=1-m_2-m_3$ in \eqref{eqn:Baltagiannis}.
One minor difference is that in \eqref{eqn:Baltagiannis} the position of $x_1$ is not constrained to be on the negative $x$-semi-axis, as we assumed for \eqref{eqn:xy_CC}; the position
of $x_1$  in \eqref{eqn:Baltagiannis} depends on the quantity $\textrm{sign}(K)$; when $\textrm{sign}(K)>0$ we have $|K|/K=1$, and  the equations \eqref{eqn:xy_CC} become equivalent with the equation  \eqref{eqn:Baltagiannis}.

We remark that when $m_3\to 0$, the limiting position of the three masses in \eqref{eqn:Baltagiannis} is given by:
\begin{equation} \label{eqn:limiting}
\begin{array}{lll}
x_1  =-m_2, & y_1  = 0, & z_1  =0,\\
x_2  =1-m_2,& y_2 =0,& z_2=0,\\
x_3= \frac{1-2m_2}{2},&y_3=\frac{\sqrt{3}}{2},&z_3=0,
\end{array}
\end{equation}
with $(x_1,y_1)$ and $(x_2,y_2)$ representing  the position of the masses $m_1$ and $m_2$, respectively,  and $(x_3,y_3)$ representing the position of the equilibrium point $L_4$  in the planar circular restricted three-body problem.
\end{rem}


\subsection{Equations of motion for the restricted four-body problem with oblate tertiary}\label{sec:4BP}
Now we consider the dynamics of a  fourth body in the neighborhood of the tertiary. This fourth body  represents the  moonlet  Skamandrios  orbiting around Hektor.  We model the dynamics of the fourth body by the spatial,
circular, restricted four-body problem, meaning that the moonlet is  moving under the gravitational attraction of Hektor, Jupiter and the Sun, without affecting their motion  which remains on circular orbits and forming a triangular central configuration as in Section \ref{sec:central_config}.
As before, we assume that Hektor has an oblate shape, with the gravitational potential given  by \eqref{eqn:C20}.

The equations of motion of the infinitesimal body relative to a synodic frame of reference that rotates together with the three bodies is given by
\begin{equation}\label{eqn:PCR4BP}\begin{split}\ddot{x}-2\omega\dot{y}&=\frac{\partial {\Tilde{\Omega}}}{\partial x}={\Tilde{\Omega}}_x\\
\ddot{y}+2\omega\dot{x}&=\frac{\partial {\Tilde{\Omega}}}{\partial x}={\Tilde{\Omega}}_y\\
\ddot{z}&=\frac{\partial {\Tilde{\Omega}}}{\partial z}={\Tilde{\Omega}}_z,
\end{split}\end{equation}
where the effective potential ${\Tilde{\Omega}}$ is given by
\[{\Tilde{\Omega}}(x,y,z) = \frac{1}{2}\omega^2(x^2+y^2)+\left(\sum_{i=1}^{3}\frac{m_i}{r_i}+ \frac{m_3}{r_3}
\left(\frac{R_3}{r_3}\right)^2\left(\frac{C_{20}}{2}\right)(3\sin^2{\phi}-1)\right)\]
with $(x_i, y_i,z_i)$ representing the $(x,y,z)$-coordinates in the synodic reference frame of the body of mass $m_i$,
$r_{i} = \left((x-x_{i})^2+(y-y_{i})^2+z^2\right)^{\frac{1}{2}}$ is the distance from the moonlet to the mass $m_i$, for $i = 1,2,3$,
and $\omega$ is the angular velocity of the system of three bodies around the center of mass given by \eqref{eqn:omega}.
The perturbed mean motion of the primaries $\omega$ in the above equation depends on the oblateness parameter.
The coordinates $(x_1, y_1)$,  $(x_2, y_2)$, $(x_3, y_3)$ of the bodies  $m_1$, $m_2$, $m_3$, respectively, are given by  \eqref{eqn:xy_CC}, while $z_i=0$, for $i=1,2,3$.

For the choice that we have made $r_{13}=r_{23}=u=1$  and $r_{12}=v$ satisfying \eqref{eqn:uv}, we have
\begin{equation}\label{omeganew}
\omega=\frac{1}{v^3}=\sqrt{1+\frac{3R_3^2J_2}{2}}=\sqrt{1-\frac{3R_3^2C_{20}}{2}}.
\end{equation}

We remark  that if we set $m_2=0$ we obtain the restricted three-body problem with one oblate body and $\omega=\sqrt{1+3R_H^2J_2/2}$ agrees with the formulas in \cite{McCuskey1963,Sharma_Rao_1976,Stoica_Arredondo_2012}.
If $m_3$ has no oblateness, i.e., $J_2=C_{20}=0$, then $\omega=1$.

We rescale the time so that in the new units the angular velocity is normalized to $1$, obtaining
\begin{equation}\label{eqn:eqnmotion}\begin{split}
\ddot{x}-2\dot{y}&=\frac{\partial{\Omega}}{\partial{x}}=\Omega_{x},\\
\ddot{y}+2\dot{x}&=\frac{\partial{\Omega}}{\partial{y}}=\Omega_{y},\\
\ddot{z}&=\frac{\partial{\Omega}}{\partial{z}}=\Omega_{z},
\end{split}\end{equation}
with
$$\Omega(x,y,z)=\frac{1}{2}(x^2+y^2)+\frac{1}{\omega^2}\left(\sum_{i=1}^{3}{ \frac{m_i}{r_i}}+ \frac{m_3}{r_3}
\left(\frac{R_3}{r_3}\right)^2\left(\frac{C_{20}}{2}\right)(3\sin^2{\phi}-1)\right).$$

The equations of motion \eqref{eqn:eqnmotion} have the total energy $H$ defined below as a conserved quantity: \marginpar{MG:trivial derivation omitted}
\begin{equation*}\begin{split}
H=&\frac{1}{2}(\dot{x}^2+\dot{y}^2+\dot{z}^2)-\Omega,
\\=&\frac{1}{2}(\dot{x}^2+\dot{y}^2+\dot{z}^2)\\&-\left[\frac{1}{2}(x^2+y^2)+
\frac{1}{\omega^2}
\left(\sum_{i=1}^{3}{\frac{m_i}{r_i}}+\frac{m_3}{r_3}\left(\frac{R_3}{r_3}\right)^2
\left(\frac{C_{20}}{2}\right)(3\sin^2{\phi}-1)\right)\right].
\end{split}
\end{equation*}

We switch to the Hamiltonian setting by considering the system of symplectic coordinates $(x,y,z,p_x,p_y,p_z)$   with respect to the symplectic form
$\varpi=x\wedge p_x+y\wedge p_y+z\wedge p_z$,
and making the transformation $\dot{x}=p_x+y$, $\dot{y}=p_y-x$ and $\dot{z}=p_z$. We obtain:
\begin{equation} \label{eq2}
\begin{split}
H= & \frac{1}{2} ((p_{x}+y)^2+(p_{y}-x)^2+p_{z}^2)-\frac{1}{2}(x^2+y^2)\\
& -\frac{1}{\omega^2}\left(\sum_{i=1}^{3}\frac{m_i}{r_i}+ \frac{m_3}{r_3}\left(\frac{R_3}{r_3}\right)^2\left(\frac{C_{20}}{2}\right)(3\sin^2{\phi}-1)\right)\\
=& \frac{1}{2}(p_{x}^2+p_{y}^2+p_{z}^2)+yp_{x}-xp_{y}\\&-\frac{1}{\omega^2}\left(\sum_{i=1}^{3}\frac{m_i}{r_i}+
\frac{m_3{R_3}^2}{r_3^3} \left(\frac{C_{20}}{2}\right)(3\sin^2{\phi}-1)\right)\\
=& \frac{1}{2}(p_{x}^2+p_{y}^2+p_{z}^2)+yp_{x}-xp_{y}\\&-\frac{1}{\omega^2}\left(\sum_{i=1}^{3}\frac{m_i}{r_i}+
\frac{m_3}{r_3^3} {C'}(3\sin^2{\phi}-1)\right),
\end{split}
\end{equation}
where we denote $C'={R_3}^2C_{20}/2$.
Thus, the equations of motion \eqref{eqn:PCR4BP}
are equivalent to the Hamilton equations for the Hamiltonian given by \eqref{eq2}.

\section{Hill  four-body problem with oblate tertiary}\label{sec:Hill}
In this section we describe a model which is obtained by taking a Hill's approximation of the spatial, circular,
restricted four-body problem with oblate tertiary.  The procedure goes as follows: we shift the origin of the coordinate system to $m_3$,  perform a rescaling of the coordinates depending
on $m_3^{1/3}$,  write the associated Hamiltonian
in the rescaled coordinates as a power series in $m_3^{1/3}$,  and
neglect  all the terms of order $O(m_3^{1/3})$  in  the expansion,
since such terms are  small when $m_3$ is small.\marginpar{MG:description changed}
 Through this procedure the masses
$m_1$ and $m_2$ are `sent to infinite distance'. This model
is an extension of the classical Hill's approximation of the restricted three-body
problem, with the major differences that our model is a four-body problem, and takes into account the  effect of the oblateness parameter $C_{20}=-J_2$ of the tertiary; compare with \cite{Hill,MEYER1982,Burgos_Gidea}.

\subsection{Hill's approximation of the restricted four-body problem with oblate tertiary  in shifted coordinates} \label{sec:Hill_shifted}


The main result is the following:

\begin{thm} \label{main theorem}
Let us consider the Hamiltonian \eqref{eq2}, let us shift the origin of the reference frame so
that it coincides with $m_3$ and let us perform the conformal symplectic scaling given by
$$(x,y,z,p_{x},p_{y},p_{z})\rightarrow m_{3}^{1/3}(x,y,z,p_{x},p_{y},p_{z}).$$
Accordingly, we rescale the average radius of the tertiary as $R_3=m_3^{1/3}\rho_3$.
Expanding the resulting  Hamiltonian
 as a power series in $m_3^{1/3}$  and
neglecting  all the terms of order $O(m_3^{1/3})$  in  the expansion, we obtain the following Hamiltonian describing the
Hill's four-body problem with oblate tertiary:\marginpar{MG:description changed}
\begin{equation}\begin{split}\label{eqn:hill_hamiltonian}
H&=\frac{1}{2}(p_x^2+p_y^2+p_z^2)+yp_x-xp_y \\
&\quad +\frac{4-3v^2}{8}x^2+ \frac{3v^2-8}{8}y^2+\frac{1}{2}z^2-\frac{3v\sqrt{4-v^2}}{4}(1-2\mu)xy
\\&\quad
-\frac{1}{ (x^2+y^2+z^2)^{\frac{1}{2}}}-\left(\frac{\rho_3^2C_{20}}{2}\right)\frac{1}{ (x^2+y^2+z^2)^{\frac{3}{2}}}
\left( \frac{3z^2}{x^2+y^2+z^2} -1\right),\\
&=\frac{1}{2}(p_x^2+p_y^2+p_z^2)+yp_x-xp_y\\
&\quad+\frac{4-3v^2}{8}x^2+ \frac{3v^2-8}{8}y^2+\frac{1}{2}z^2-\frac{3v\sqrt{4-v^2}}{4}(1-2\mu)xy
\\&\quad
-\frac{1}{ (x^2+y^2+z^2)^{\frac{1}{2}}}-\frac{{\cc } }{ (x^2+y^2+z^2)^{\frac{3}{2}}}
\left( \frac{3z^2}{x^2+y^2+z^2} -1\right),
\end{split}
\end{equation}
where $v=\left(\frac{1}{1-\frac{3}{2}R_3^2C_{20}} \right)^{1/3}$, $\mu=\frac{m_2}{m_1+m_2}$, and  $\cc :=\rho_3^2C_{20}/2=m_3^{-\frac{2}{3}}R^2_3C_{20}/2$.
\end{thm}
\marginpar{MG: notation $c_{20}$ replaced by $c$ to avoid confusion; Hamiltonian in terms of $c$ added.}
\begin{proof}
We start by shifting the origin of the coordinate system $(x,y,z)$ to the location of the mass $m_3$ (representing Hektor), via the change of coordinates
\[
\begin{array}{lll}
\xi=x-x_3,     & \eta=y-y_3,     & \zeta=z,\\
p_\xi=p_x+y_3, & p_\eta=p_y-x_3, & p_\zeta=p_z .
\end{array}
\]

The Hamiltonian corresponding to \eqref{eq2} becomes
\begin{equation} \label{eq3}
\begin{split}
H =& \frac{1}{2}[(p_{\xi}-y_3)^2+(p_{\eta}+x_3)^2+p_{\zeta}^2]\\&+(\eta+y_3)(p_{\xi}-y_3)-(\xi+x_3)(p_{\eta}+x_3)\\
& -\frac{1}{\omega^2}\left(\sum_{i=1}^{3}\frac{m_i}{\bar{r}_i}+\frac{m_3}{\bar{r}_3}
\left(\frac{R_3}{\bar{r_3}}\right)^2\left(\frac{C_{20}}{2}\right)(3\sin^2{\phi}-1)\right)\\
=& \frac{1}{2} [(p_{\xi}^2-2p_{\xi}y_3+y_3^2)+(p_{\eta}^2+2p_{\eta}x_3+x_3^2)+p_{\zeta}^2]\\
&+\eta p_{\xi}-\eta y_3+y_3p_\xi-y_3^2-\xi p_\eta-\xi x_3-x_3p_\eta-x_3^2\\
&-\frac{1}{\omega^2}\left(\sum_{i=1}^{3}\frac{m_i}{\bar{r}_i}+\frac{m_3}{\bar{r}_3}
\left(\frac{R_3}{\bar{r}_3}\right)^2\left(\frac{C_{20}}{2}\right)(3\sin^2{\phi}-1)\right) \\
=&\frac{1}{2} (p_{\xi}^2+p_{\eta}^2+p_{\zeta}^2)+\eta p_{\xi}-\xi p_\eta-(\xi x_3+\eta y_3)\\
& -\frac{1}{\omega^2}\left(\sum_{i=1}^{3}\frac{m_i}{\bar{r}_i}+
\frac{m_3}{\bar{r}_3}\left(\frac{R_3}{\bar{r}_3}\right)^2\left(\frac{C_{20}}{2}\right)(3\sin^2{\phi}-1)\right)\\&-\frac{1}{2}(x_3^2+y_3^2),
\end{split}
\end{equation}
where $\bar{r}_i^2=(\xi-\bar{x}_i)^2+(\eta-\bar{y}_i)^2+\zeta^2=(\xi+x_3-x_i)^2+(\eta+y_3-y_i)^2+\zeta^2$, with  $\bar{x}_i=x_i-x_3$,
$\bar{y}_i=y_i-y_3$. Note that $\bar{r}_3=r_3$.
Since $-\frac{1}{2}(x_3^2+y_3^2)$ is a constant term, it plays no role in the Hamiltonian equations and it will be dropped in the following calculation.

Since $\sin{\phi}=\frac{z}{\bar{r}_3}$, we have
\begin{equation} \label{*}
\begin{split}
H =&\frac{1}{2} (p_{\xi}^2+p_{\eta}^2+p_{\zeta}^2)+\eta p_{\xi}-\xi p_{\eta}-(\xi x_3+\eta y_3)\\
&-\frac{1}{\omega^2}\left[\sum_{i=1}^{3}\frac{m_i}{\bar{r}_i}+\frac{m_3}{\bar{r}_3}\left(\frac{R_3}{\bar{r}_3}\right)^2
\left(\frac{C_{20}}{2}\right)\left(3\left(\frac{\zeta}{\bar{r}_3}\right)^2-1\right)\right].
\end{split}
\end{equation}
\marginpar{MG:$x$, $y$, $z$ replaced with $\xi$, $\eta$, $\zeta$}

We now perform the following conformal symplectic scaling  with multiplier $m_3^{-2/3}$,\marginpar{MG: the multiplier is the factor that appears in front of the Hamiltonian (3.6), see Meyer and Hall.}
given by (with a little abuse of notation, we call again the new variables $x$, $y$, $z$, $p_x$,
$p_y$, $p_z$):
\begin{equation}
\begin{split}
\xi&=m_3^{\frac{1}{3}} x\ ,\qquad \eta=m_3^{\frac{1}{3}} y\ ,\qquad \zeta=m_3^{\frac{1}{3}}z,\\
p_\xi&=m_3^{\frac{1}{3}} p_x\ ,\qquad p_\eta=m_3^{\frac{1}{3}} p_y\ ,\qquad p_\zeta=m_3^{\frac{1}{3}} p_z\ .
\end{split}
\end{equation}
Consistently with this scale change, we also introduce the scaling transformation of
the average radius of the smallest body
\begin{equation}\label{eqn:J2rescaling}
R_3^2 =(m_3^{1/3}\rho_3)^2 =m_3^{2/3}\rho_3^2.
\end{equation}

The choice of the power of $m_3$ is motivated by the fact that in this way the gravitational force becomes
of the same order of the centrifugal and Coriolis forces (see, e.g., \cite{MEYER1982}).

The purpose of the subsequent calculation is that, after the above substitutions, we expand the resulting  Hamiltonian
 as a power series in $m_3^{1/3}$  and
neglect  all the terms of order $O(m_3^{1/3})$  in  the expansion.\marginpar{MG:sentence changed}

We expand  the terms $\displaystyle\frac{1}{\bar{r}_1}$ and $\displaystyle\frac{1}{\bar{r}_2}$ in Taylor series around the new origin of coordinates,  obtaining
\begin{equation*}\begin{split}f^1:=\frac{1}{\bar{r}_1}=\sum_{k\geq 0}{P_k^1(x,y,z)},\\
f^2:=\frac{1}{\bar{r}_2}=\sum_{k\geq 0}{P_k^2(x,y,z)},\end{split}\end{equation*}
where $P_k^j(x,y,z)$ is a homogeneous polynomial of degree $k$, for $j=1,2$.

The resulting Hamiltonian takes the form:
\begin{equation} \label{eq4}
\begin{split}
H =& m_3^{-\frac{2}{3}}\left[ \frac{1}{2} (m_3^{\frac{2}{3}}p_x^2+m_3^{\frac{2}{3}}p_y^2+m_3^{\frac{2}{3}}p_z^2) \right.
\\&\qquad+m_3^{\frac{2}{3}}yp_x-
m_3^{\frac{2}{3}}xp_y-m_3^{\frac{1}{3}}xx_3-m_3^{\frac{1}{3}}yy_3\\
&\qquad -\frac{1}{\omega^2}\left(\sum_{k\geq 1}{m_1m_3^{\frac{k}{3}}P_k^1(x,y,z)}+\sum_{k \geq 1}{m_2m_3^{\frac{k}{3}}P_k^2(x,y,z)}\right.\\
&\left.\left.\qquad +\frac{m_3^{\frac{2}{3}}}{\bar{r}_3}+\frac{m_3^{\frac{2}{3}}}{\bar{r}_3}\left(\frac{\rho_3}{\bar{r}_3}\right)^2
\left(\frac{C_{20}}{2}\right)\left(3\left(\frac{z}{\bar{r}_3}\right)^2-1\right)\right)\right] \\
=& \frac{1}{2}(p_x^2+p_y^2+p_z^2)+yp_x-xp_y-m_3^{-\frac{1}{3}}xx_3-m_3^{-\frac{1}{3}}yy_3\\
&\quad -\frac{1}{\omega^2}\left (m_3^{-\frac{1}{3}} m_1 P_1^1+m_3^{-\frac{1}{3}} m_2 P_1^2\right. \\
&\qquad\qquad +\sum_{k\geq2}m_3^{\frac{k-2}{3}}m_1P_k^1(x,y,z)+\sum_{k\geq2}m_3^{\frac{k-2}{3}}m_2P_k^2(x,y,z)
\\ &\qquad\qquad\left. +\frac{1}{\bar{r}_3}+\frac{1}{\bar{r}_3}\left(\frac{\rho_3^2}{\bar{r}_3^2}\right)\left(\frac{C_{20}}{2}\right)\left(3\left(\frac{z}{\bar{r}_3}\right)^2-1\right)\right)\\
=& \frac{1}{2}(p_x^2+p_y^2+p_z^2)+yp_x-xp_y\\
&\quad -m_3^{-\frac{1}{3}}\left(xx_3+yy_3+\frac{m_1 P_1^1}{\omega^2} +\frac{m_2 P_1^2 }{\omega^2}\right)\\
&\quad  -\frac{1}{\omega^2}\left(\sum_{k\geq2}{m_3^{\frac{k-2}{3}}m_1P_k^1(x,y,z)}+
\sum_{k\geq2}{m_3^{\frac{k-2}{3}}m_2P_k^2(x,y,z)} \right.\\
 & \qquad\qquad\left . +\frac{1}{\bar{r}_3} +\frac{\rho_3^2}{\bar{r}_3^3}\left(\frac{C_{20}}{2}\right) \left(3\left(\frac{z}{\bar{r}_3}\right)^2-1\right)\right).
\end{split}
\end{equation}
In the following, we will disregard in \eqref{eq4} all terms which are of order of $m_3^{1/3}$, as in  classical Hill's theory of lunar motion (\cite{MEYER1982}).\marginpar{MG: sentence changed}

We compute the first-degree polynomials $P^i_1$, for $i=1,2$,
\begin{equation}\label{eqn:hill-0} \begin{split}P^i_1&=\frac{\partial f^i}{\partial x}(0,0,0)x
+\frac{\partial f^i}{\partial y}(0,0,0)y+\frac{\partial f^i}{\partial z}(0,0,0)z=\frac{\bar{x}_i}{r_{i3}^3}x+\frac{\bar{y}_i}{r_{i3}^3}y\\&=(x_i-x_3)x+(y_i-y_3)y,
\end{split}
\end{equation}
where $r_{i3}=\sqrt{(x_i-x_3)^2+(y_i-y_3)^2}=u=1$ represents the distance between the mass $m_i$ and $m_3$.

We now compute the contribution of the different terms in \eqref{eq4}.
Using that  $m_1+m_2+m_3=1$, and that $m_1x_1+m_2x_2+m_3x_3=m_1y_1+m_2y_2+m_3y_3=0$,
from \eqref{eqn:4}, \eqref{eqn:5}, \eqref{eqn:6}, we obtain that
\begin{equation}
\label{eqn:poly1}\begin{split}
m_3^{-\frac{1}{3}}&(xx_3+yy_3+ \frac{m_1 P_1^1}{\omega^2} +\frac{m_2 P_1^2}{\omega^2})\\
&=m_3^{-\frac{1}{3}}\left.[x(x_3+\frac{m_1}{\omega^2}(x_1-x_3) + \frac{m_2}{\omega^2}(x_2-x_3))\right.\\
& \left.\qquad\quad+y(y_3+\frac{m_1}{\omega^2}(y_1-y_3) + \frac{m_2}{\omega^2}(y_2-y_3))\right]\\
&=m_3^{-\frac{1}{3}}\left[x(x_3+\frac{1}{\omega^2}(m_1x_1+m_2x_2+m_3x_3- x_3))\right.\\
&\,\left.\qquad\quad+y(y_3+\frac{1}{\omega^2}(m_1y_1+m_2y_2+m_3y_3- y_3))\right]\\
&=m_3^{-\frac{1}{3}}\left(1-\frac{1}{\omega^2}\right)(xx_3+yy_3)\\
&=-m_3^{-\frac{1}{3}}\frac{\frac{3}{2}m_3^{\frac{2}{3}}\rho_3^2C_{20}}{1-\frac{3}{2}m_3^{\frac{2}{3}}\rho_3^2C_{20}}
(xx_3+yy_3)\\
&= -m_3^{\frac{1}{3}}\frac{\frac{3}{2}\rho_3^2C_{20}}{1-\frac{3}{2}m_3^{\frac{2}{3}}\rho_3^2C_{20}}(xx_3+yy_3),
\end{split}
\end{equation}
where we used \eqref{omeganew} and \eqref{eqn:J2rescaling}.
The expression in \eqref{eqn:poly1} is $O(m_3^{1/3})$ so it will be omitted in the Hill approximation.\marginpar{MG:sentence changed}

We compute the second degree polynomials $P^i_2$, for $i=1,2$,
\begin{equation}
\label{eqn:hill-3}\begin{split}
P^i_2&=\frac{1}{2}\left(\frac{\partial^2 f^i}{\partial x^2}(0,0,0)x^2+\frac{\partial^2 f^i}{\partial y^2}(0,0,0)y^2
+\frac{\partial^2 f^i}{\partial z^2}(0,0,0)z^2\right)\\
&\quad+\left(\frac{\partial^2 f^i}{\partial x\partial y}(0,0,0)xy+\frac{\partial^2 f^i}{\partial y\partial z}(0,0,0)yz
+\frac{\partial^2 f^i}{\partial z\partial x}(0,0,0)zx\right)\\&=\frac{1}{2}\left (\frac{3\bar{x}^2_i}{r_{i3}^5}-\frac{1}{r_{i3}^3}\right)x^2+\frac{1}{2}\left (\frac{3\bar{y}^2_i}{r_{i3}^5}-\frac{1}{r_{i3}^3}\right)y^2+\frac{1}{2}\left (-\frac{1}{r_{i3}^3}\right)z^2\\&\,+\left(\frac{3\bar{x}_i\bar{y}_i}{r_{i3}^5}\right)xy\\
&=\frac{1}{2}\left ( {3\bar{x}^2_i} -1\right)x^2+\frac{1}{2}\left ( {3\bar{y}^2_i} -1\right)y^2-\frac{1}{2}z^2+\left({3\bar{x}_i\bar{y}_i}\right) xy,
\end{split}
\end{equation}
since $r_{13}=r_{23}=u=1$.

The corresponding terms in \eqref{eq4} yield
\begin{equation}\label{eqn:poly2}\begin{split}
\frac{1}{\omega^2}&(m_1P^1_2+m_2P^2_2)\\
&=\frac{1}{\omega^2}  \Big[\frac{1}{2}\left((1-m_2-m_3)(3(x_1-x_3)^2-1) +m_2(3(x_2-x_3)^2-1)\right)x^2\\
&\qquad +\frac{1}{2}\left((1-m_2-m_3)(3(y_1-y_3)^2-1) +m_2(3(y_2-y_3)^2-1)\right)y^2\\
&\qquad + \frac{1}{2}\left(-1+m_3\right)z^2\\
&\qquad + 3\left((1-m_2-m_3)(x_1-x_3)(y_1-y_3)+m_2(x_2-x_3)(y_2-y_3)\right)xy\Big].
\end{split}
\end{equation}
Using \eqref{eqn:xy_CC_m3_zero} and that
\[\frac{1}{\omega^2}=\frac{1}{1+3C}=\frac{1}{1-\frac{3}{2}m_3^{\frac{2}{3}}\rho_3^2C_{20}}=1+O(m_3^{1/3}),\]
omitting the terms of order $O(m_3^{1/3})$, \marginpar{MG: sentence changed}
the expression \eqref{eqn:poly2} becomes
\[ \frac{3v^2-4}{8}x^2+ \frac{8-3v^2}{8}y^2-\frac{1}{2}z^2+\frac{3v\sqrt{4-v^2}}{4}(1-2m_2)xy.\]
The expression in  \eqref{eq4}  contains terms of the form
\[\sum_{k\geq 3}m_3^{\frac{k-2}{3}}m_1P_k^1(x,y,z)+\sum_{k\geq 3}m_3^{\frac{k-2}{3}}m_2P_k^2(x,y,z),\]
which can be written in terms of positive powers of $m_3^{1/3}$, which are omitted in the Hill approximation. \marginpar{MG: sentence changed}

The remaining terms in \eqref{eq4} are
\[\frac{1}{\bar{r}_3}
+\left(\frac{\rho_3^2}{\bar{r}_3^3}\right)\left(\frac{C_{20}}{2}\right)
\left(3\left(\frac{z}{\bar{r}_3}\right)^2-1\right) \]
and they do not depend on $m_3$.

Therefore, when we omit all terms of order $O(m_3^{1/3})$ in \eqref{eq4}, and taking into account that $\frac{1}{\omega^2}=1+O(m_3^{1/3})$, we obtain the following  Hamiltonian\marginpar{MG:sentence changed}
\begin{equation}\label{eqn:Hill-Ham}\begin{split}
H&=\frac{1}{2}(p_x^2+p_y^2+p_z^2)+yp_x-xp_y \\
&\quad -\frac{3v^2-4}{8}x^2- \frac{8-3v^2}{8}y^2+\frac{1}{2}z^2-\frac{3v\sqrt{4-v^2}}{4}(1-2m_2)xy
\\&\quad
-\frac{1}{ (x^2+y^2+z^2)^{\frac{1}{2}}}-\frac{\rho_3^2}{ (x^2+y^2+z^2)^{\frac{3}{2}}}\left(\frac{C_{20}}{2}\right)
\left( \frac{3z^2}{x^2+y^2+z^2} -1\right)\\
&=\frac{1}{2}(p_x^2+p_y^2+p_z^2)+yp_x-xp_y\\
&\quad -\frac{3v^2-4}{8}x^2- \frac{8-3v^2}{8}y^2+\frac{1}{2}z^2-\frac{3v\sqrt{4-v^2}}{4}(1-2\mu)xy
\\&\quad
-\frac{1}{ (x^2+y^2+z^2)^{\frac{1}{2}}}-\frac{{\cc } }{ (x^2+y^2+z^2)^{\frac{3}{2}}}
\left( \frac{3z^2}{x^2+y^2+z^2} -1\right),
\end{split}
\end{equation}
where we used $\mu=m_2/(m_1+m_2)$ and  $\cc :=\rho_3^2C_{20}/2=m_3^{-\frac{2}{3}}R^2_3C_{20}/2$.
\end{proof}

We remark that a similar strategy was adopted in \cite{Markellos}, where a Hill's three body problem with oblate
primaries has been considered.

We refer to the Hamiltonian in \eqref{eqn:Hill-Ham} as the \sl Hill's approximation. \rm
It can be thought of as the limiting Hamiltonian, when the primary and the secondary are sent at an infinite distance, and their total mass becomes infinite.
It provides an approximation of the motion of the infinitesimal particle in an $O(m_3^{1/3})$
neighborhood of $m_3$.  Remarkably, the angular velocity $\omega$ does not appear in the limiting Hamiltonian.

We introduce the gravitational potential as
\begin{equation}
\begin{split}
\widehat U(x,y,z)&=
\frac{3v^2-4}{8}x^2+ \frac{8-3v^2}{8}y^2-\frac{1}{2}z^2+\frac{3v\sqrt{4-v^2}}{4}(1-2\mu)xy
\\&\quad
+\frac{1}{ (x^2+y^2+z^2)^{\frac{1}{2}}}+\frac{{\cc }}{ (x^2+y^2+z^2)^{\frac{3}{2}}}
\left( \frac{3z^2}{x^2+y^2+z^2} -1\right),
\end{split}
\end{equation}
and the effective potential as
\begin{equation}
\begin{split}
\widehat\Omega(x,y,z) &=\frac{1}{2}(x^2+y^2)+\widehat U(x,y,z) \\
&=
\frac{3v^2}{8}x^2+ \frac{3(4-v^2)}{8}y^2-\frac{1}{2}z^2+\frac{3v\sqrt{4-v^2}}{4}(1-2\mu)xy\\&\quad
+\frac{1}{ (x^2+y^2+z^2)^{\frac{1}{2}}}+\frac{\cc }{ (x^2+y^2+z^2)^{\frac{3}{2}}}
\left( \frac{3z^2}{x^2+y^2+z^2} -1\right).
\end{split}
\end{equation}
The equations of motion associated to \eqref{eqn:Hill-Ham} can thus be written as:
\begin{equation*}\begin{split}\ddot{x}-2\dot{y}&=\widehat\Omega_x,\\
\ddot{y}+2\dot{x}&=\widehat\Omega_y,\\
\ddot{z}&=\widehat\Omega_z.\end{split}\end{equation*}

\begin{rem}\label{rem:L4L5}
In the case when $C_{20}=0$, we have that $v=1$ and  the Hamiltonian in \eqref{eqn:Hill-Ham} is the same as the one obtained in \cite{Burgos_Gidea}. Also, its quadratic part  coincides with the quadratic part of the expansion of the Hamiltonian of the restricted three-body problem centered at the Lagrange libration point $L_{4}$.
Moreover, in the special case $\mu=0$  we obtain  the classical lunar Hill's problem after some coordinate transformation
(see Section~\ref{sec:Hill_system}).
\end{rem}

\subsection{Hill's four-body model applied to the Sun-Jupiter-Hektor system}\label{sec:Hill_system} In the case of the Sun-Jupiter-Hektor system, using the data from Section \ref{section:data} in the above equations we obtain $\mu=m_2/(m_1+m_2)=1.898\times10^{27}/(1.989\times10^{30}+1.898\times10^{27})=0.0009533386$.
Also, for Hektor we have $C_{20}=-0.476775$, the average radius of Hektor is $R_3=92$ km, and the mass of Hektor is $m_H=7.91\times10^{18}$ kg. In the normalized units, where we use the average distance Sun-Jupiter $778.5\times 10^6$ km as the unit of distance, and the mass of Sun-Jupiter-Hektor $7.91\times10^{18}+1.989\times10^{30}+1.898\times10^{27}=1.990898\times 10^{30}$ kg
as the unit of mass, we have that the normalized average radius of Hektor is  $R_3=92/(778.5\times{10}^{6})=1.18176\times 10^{-7}$
and  the normalized mass of Hektor  is $m_3=7.91\times10^{18}/1.990898\times 10^{30}=3.97308\times10^{-12}$.
Hence, we obtain
\begin{equation}\label{eqn:normalized_\cc }
\cc =m_3^{-\frac{2}{3}}R^2_3C_{20}/2=-1.32716\times 10^{-7}.
\end{equation}
Also, $\rho_3= m_3^{-\frac{1}{3}}R_3= 0.000746$.

We note that if we consider the restricted four-body problem (without the Hill's approximation) described
by the Hamiltonian \eqref{eq2}, the oblateness effect is given by the coefficient
$C'=R_3^2C_{20}/2=-3.32921544\times 10^{-15}$, which is much smaller then
$\cc $ in \eqref{eqn:normalized_\cc }. As expected, the Hill's approximation is acting like a `magnifying glass' of the dynamics in a neighborhood of Hektor.

\subsection{Hill's approximation of the restricted four-body problem with oblate tertiary in rotated coordinates}
In this section we write the Hamiltonian of Hill's approximation of the
restricted four-body problem with oblate tertiary in a rotating reference frame
in which the primary and the secondary will be located on the horizontal axis.

\begin{cor}\label{cor:Hill}
The Hamiltonian   \eqref{eqn:hill_hamiltonian} is equivalent, via a rotation of the coordinate axes that
places the primary and the secondary on the $x$-axes, to the Hamiltonian
\begin{equation}\label{eqn:hill_rotated}\begin{split}
H&=\frac{1}{2}(p_x^2+p_y^2+p_z^2)+y p_x-xp_y\\
&+\left(\frac{1-\lambda_2}{2}\right)x^2+\left(\frac{1-\lambda_1}{2}\right)y^2+\frac{1}{2}z^2\\
&-\frac{1}{\sqrt{x^2+y^2+z^2}}-\left(\frac{\rho_3^2C_{20}}{2}\right)
\frac{1}{(x^2+y^2+z^2)^{\frac{3}{2}}}\left (\frac{3z^2}{x^2+y^2+z^2} -1\right),
\end{split}\end{equation}
where $\lambda_2$ and $\lambda_1$ are the eigenvalues corresponding to the rotation transformation
in the $xy$-plane, and $\rho_3=m_3^{-1/3}R_3$.
\end{cor}
\begin{proof}
We perform a  rotation on the $xy-$plane and re-write the Hamiltonian in \eqref{eqn:Hill-Ham}
in the framework of the rotating coordinates, which are more suitable for the subsequent analysis.
Since the rotation will be performed on the plane, we restrict the computations to the planar case. The planar effective potential restricted to the $xy-$plane is given by
\begin{equation*}\begin{split}\widehat\Omega (x,y)=&\frac{3v^2}{8}x^2+\frac{3v\sqrt{4-v^2}}{4}(1-2\mu)xy+\frac{3(4-v^2)}{8}y^2\\
&+\frac{1}{(x^2+y^2)^{1/2}}-  \frac{\cc }{(x^2+y^2)^{3/2}},\end{split}\end{equation*}
which can be written in matrix rotation as
$$\widehat\Omega = \frac{1}{2}w^TMw+\frac{1}{\lVert w \rVert}-  \frac{\cc }{ \lVert w \rVert^3},$$
where $w=(x,y)^T$ and
\[ M=
  \left[ {\begin{array}{ll}
   \frac{3v^2}{4} & \frac{3v\sqrt{4-v^2}}{4}(1-2\mu) \\[0.5em]
     \frac{3v\sqrt{4-v^2}}{4}(1-2\mu)  & \frac{3(4-v^2)}{4} \end{array} } \right].
\]

Notice that the matrix $M$ is symmetric, so its eigenvalues are real, the eigenvectors $v_1$ and $v_2$ are orthogonal, and the corresponding orthogonal matrix $C=\textrm{col}(v_2,v_1)$
defines a rotation in the $xy$-plane.
We find the eigenvalues of $M$ by solving the characteristic equation:
\begin{equation} \label{eq18}
\begin{split}
\det( M-\lambda I )=0 \Rightarrow& \lambda^2-3\lambda +\frac{9v^2(4-v^2)}{4}(\mu-\mu^2)=0,\\
\Rightarrow& \lambda_{1}=\frac{3- 3\sqrt{1-v^2(4-v^2)(\mu-\mu^2)}}{2},\\& \lambda_{2}=\frac{3+ 3\sqrt{1-v^2(4-v^2)(\mu-\mu^2)}}{2}.
\end{split}
\end{equation}
Since $0<\mu\leq \frac{1}{2}$, $\mu-\mu^2\leq 1/4$, and we have $1>1-v^2(4-v^2)(\mu-\mu^2)\geq 1-\frac{1}{4}v^2(4-v^2)=\left(1-\frac{v^2}{2}\right)^2> 0$. Thus $\lambda_1,\lambda_2>0$ and $\lambda_1\neq\lambda_2$.

We notice that when $\mu=\frac{1}{2}$, the matrix $M$ is already a diagonal matrix, and the corresponding
eigenvalues are $\lambda_1=\frac{3v^2}{4} $ and $\lambda_2=\frac{3(4-v^2)}{4}$.
Therefore, below we consider the case $\mu\neq\frac{1}{2}$ for which we proceed to compute
the eigenvectors associated to $\lambda_1$ and $\lambda_2$.

The eigenvector $v_1$ such that $Mv_1=\lambda_1v_1$ and $\lVert v_1 \rVert=1$ is given by
\[
   v_1=
  \left[ {\begin{array}{c}
   \displaystyle\frac{\frac{3(4-v^2)}{4} -\lambda_1}{\Delta_1}  \\[0.75em]
      \displaystyle-\frac{\frac{3v\sqrt{4-v^2}}{4}(1-2\mu)}{\Delta_1}     \end{array} } \right],
\]
where \[\Delta_1 = \left[\frac{9v^2(4-v^2)}{8}(1-2\mu)^2+\frac{3(v^2-2)}{2}\left(\lambda_1-\frac{3(4-v^2)}{4}\right)\right]^{1/2}.\]

Similarly, the eigenvector $v_2$ such that $Mv_2=\lambda_2v_2$ and $\lVert v_2 \rVert=1$ is given by
\[
   v_2=
  \left[ {\begin{array}{c}
   \displaystyle\frac{\frac{3(4-v^2)}{4} -\lambda_2}{\Delta_2}  \\[0.75em]
      \displaystyle-\frac{\frac{3v\sqrt{4-v^2}}{4}(1-2\mu)}{\Delta_2}     \end{array} } \right],
\]
where \[\Delta_2 = \left[\frac{9v^2(4-v^2)}{8}(1-2\mu)^2+\frac{3(v^2-2)}{2}\left(\lambda_2-\frac{3(4-v^2)}{4}\right)\right]^{1/2}.\]

The equations of motion for the planar case can be written as \\
$$\ddot{w}-2\mathcal{J}\dot{w}=Mw-\frac{w}{\lVert w \rVert^3}+\frac{{3\cc }w}{\lVert w \rVert^5},$$
where
\[
   \mathcal{J}=
  \left[ {\begin{array}{cc}
   0 & 1 \\      -1 & 0 \      \end{array} } \right].
\]
Consider the linear change of variable $w=C\bar{w}$ with $\bar{w} = (\bar{x},\bar{y})^T$. By substituting the new variable and multiplying $C^{-1}$ from the left, we obtain\\
$$C^{-1}C\ddot{\bar{w}}-2C^{-1}\mathcal{J} C\dot{\bar{w}}=C^{-1}MC\bar{w}-\frac{C^{-1}C\bar{w}}{\lVert \bar{w} \rVert^3}+3\cc \frac{C^{-1}C\bar{w}}{\lVert \bar{w} \rVert^5}.$$
Notice that $D=C^{-1}MC$ is the diagonal  matrix  $D=\textrm{diag}(\lambda_2,\lambda_1)$, that is $\lVert C \bar{w} \rVert^3=\lVert \bar{w} \rVert^3$. Therefore the equation becomes \\
$$\ddot{\bar{w}}-2C^{-1}\mathcal{J}C\dot{\bar{w}}=D\bar{w} -\frac{\bar{w}}{\lVert \bar{w} \rVert^3}+\frac{3\cc \bar{w}}
{\lVert \bar{w} \rVert^5}.$$
Recall that $v_1=(v_{11},v_{12})^T$, $v_2=(v_{21},v_{22})^T$ and $C=\textrm{col}(v_2,v_1)$. Since $C$ is unitary,
we have $C^{-1}=C^{T}$ and moreover
\[
   C^{-1}\mathcal{J} C=
  \left[ {\begin{array}{cc}
   0 & v_{12}v_{21}-v_{11}v_{22} \\
   -(v_{12}v_{21}-v_{11}v_{22}) & 0 \      \end{array} } \right].
\]
A direct computation shows that $v_{12}v_{21}-v_{11}v_{22}=1$, which implies
$C^{-1}\mathcal{J} C=\mathcal{J}$. Since $C^{-1}\mathcal{J}C=C^T\mathcal{J}C=\mathcal{J}$, the matrix $C$ is symplectic by definition. Therefore, the change of coordinates is symplectic.
Thus, the equations of motion can be written as \\
$$\ddot{\bar{w}}-2\mathcal{J}\dot{\bar{w}}=D\bar{w}-\frac{\bar{w}}{\lVert \bar{w} \rVert^3}+
\frac{3\cc \bar{w}}{\lVert \bar{w} \rVert^5}.$$
For $\mu \in [0,\frac{1}{2})$, we obtain the equations
\begin{equation} \label{eq19}
\begin{split}
&\ddot{\bar{x}}-2\dot{\bar{y}}=\bar{\Omega}_{\bar{x}}\\
&\ddot{\bar{y}}+2\dot{\bar{x}}=\bar{\Omega}_{\bar{y}}
\end{split}
\end{equation}
with $$\bar{\Omega}(\bar{x},\bar{y}) = \frac{1}{2}(\lambda_2\bar{x}^2+\lambda_1\bar{y}^2)+\frac{1}{\lVert \bar{w} \rVert}-
\frac{\cc }{\lVert \bar{w} \rVert^3}.$$
From the expressions for $\bar{\Omega}_{\bar{x}}$ and $\bar{\Omega}_{\bar{y}}$, we notice the symmetry properties:
$$\bar{\Omega}_{\bar{x}}(\bar{x},-\bar{y})=\bar{\Omega}_{\bar{x}}(\bar{x},\bar{y})\ ,\qquad
\bar{\Omega}_{\bar{y}}(\bar{x},-\bar{y})=-\bar{\Omega}_{\bar{x}}(\bar{x},\bar{y}).$$
Using these properties, we see that the equations (\ref{eq19}) are invariant under the transformations
$\bar{x}\rightarrow \bar{x}$, $\bar{y} \rightarrow -\bar{y}$, $\dot{\bar x}\rightarrow -\dot{\bar x}$, $\dot{\bar y} \rightarrow \dot{\bar y}$,
$\ddot{\bar x} \rightarrow \ddot{\bar x}$ and $\ddot{\bar y}\rightarrow -\ddot{\bar y}$.

If we now go back to the spatial problem, we need to replace $\bar\Omega$ by
\begin{equation}\label{eqn:eff_poten_rot}\begin{split}\bar{\Omega}(\bar{x},\bar{y},\bar{z})=
&\frac{1}{2}(\lambda_2\bar{x}^2+\lambda_1\bar{y}^2-\bar{z}^2)\\
&+\frac{1}{(\bar{x}^2+\bar{y}^2+\bar{z}^2)^{\frac{1}{2}}}-
\frac{\cc }{(\bar{x}^2+\bar{y}^2+\bar{z}^2)^{\frac{3}{2}}}
+\frac{3\cc \bar{z}^2}{(\bar{x}^2+\bar{y}^2+\bar{z}^2)^{\frac{5}{2}}};
\end{split}\end{equation}
writing
$\bar{\Omega}(\bar{x},\bar{y},\bar{z})=\frac{1}{2}\bar{x}^2+\frac{1}{2}\bar{y}^2+\bar U(\bar{x},\bar{y},\bar{z})$, we can define $\bar U$ as
\begin{equation}\label{eqn:grav_poten_rot}
\begin{split}
\bar{U}(\bar{x},\bar{y},\bar{z})=&\bar{\Omega}(\bar{x},\bar{y},\bar{z})-\frac{1}{2}\bar{x}^2-\frac{1}{2}\bar{y}^2\\
=&\left(\frac{\lambda_2-1}{2}\right)\bar{x}^2+\left(\frac{\lambda_1-1}{2}\right)\bar{y}^2-\frac{1}{2}\bar{z}^2\\
&+\frac{1}{(\bar{x}^2+\bar{y}^2+\bar{z}^2)^{\frac{1}{2}}}  -
\frac{\cc }{ (\bar{x}^2+\bar{y}^2+\bar{z}^2)^{\frac{3}{2}}}
+\frac{3\cc \bar{z}^2}{(\bar{x}^2+\bar{y}^2+\bar{z}^2)^{\frac{5}{2}}}.
\end{split}
\end{equation}
In conclusion, the Hamiltonian in these new coordinates is given by the following expression
(we omit the bars for $x$ and $y$ to simplify the notation):
\begin{equation*}
\begin{split}
H(x,y,z,p_x,p_y,p_z) &=\frac{1}{2}(p_x^2+p_y^2+p_z^2)+y p_x-xp_y\\
&+\left(\frac{1-\lambda_2}{2}\right)x^2+\left(\frac{1-\lambda_1}{2}\right)y^2+\frac{1}{2}z^2\\
&-\frac{1}{(x^2+y^2+z^2)^{\frac{1}{2}}}+\frac{\cc }{(x^2+y^2+z^2)^{\frac{3}{2}}}
-\frac{3\cc z^2}{(x^2+y^2+z^2)^{\frac{5}{2}}},
\end{split}
\end{equation*}
which coincides with \eqref{eqn:hill_rotated}.
\end{proof}

\begin{rem}\label{rem:L4L5}
If we let $C_{20}=0$ and $\mu=0$  in \eqref{eqn:hill_rotated}, we obtain  the Hamiltonian for the classical lunar Hill problem, see, e.g., \cite{MEYER1982}.
\end{rem}

\section{Linear stability analysis of the Hill  four-body problem with oblate
tertiary}\label{sec:linear}
In this section we determine the
equilibrium points associated to the potential in
\eqref{eqn:eff_poten_rot} and we analyze their linear stability.

\subsection{The equilibrium points of the system}
Considering the  potential \eqref{eqn:eff_poten_rot} (again we omit the bars for a simplified notation),
\begin{equation*}\begin{split}\Omega=&\frac{1}{2}(\lambda_2x^2+\lambda_1y^2-z^2)
\\&+\frac{1}{(x^2+y^2+z^2)^{\frac{1}{2}}}
-\frac{\cc }{(x^2+y^2+z^2)^{\frac{3}{2}}}
+\frac{3\cc z^2}{(x^2+y^2+z^2)^{\frac{5}{2}}},
\end{split}\end{equation*}
we compute its derivatives as
\begin{equation}
\begin{split}
\Omega_x &=\lambda_2x-\frac{x}{(x^2+y^2+z^2)^{\frac{3}{2}}}+\frac{3\cc x}{(x^2+y^2+z^2)^{\frac{5}{2}}}
-\frac{15\cc z^2x}{(x^2+y^2+z^2)^{\frac{7}{2}}}\\
& =\lambda_2x-\frac{x}{r^3}+\frac{3\cc x}{r^5}-\frac{15\cc z^2x}{r^7}\\
\Omega_y &=\lambda_1y-\frac{y}{(x^2+y^2+z^2)^{\frac{3}{2}}}+\frac{3\cc y}{(x^2+y^2+z^2)^{\frac{5}{2}}}
-\frac{15\cc z^2y}{(x^2+y^2+z^2)^{\frac{7}{2}}}\\
& =\lambda_1y-\frac{y}{r^3}+\frac{3\cc y}{r^5}-\frac{15\cc z^2y}{r^7}\\
\Omega_z &=-z-\frac{z}{(x^2+y^2+z^2)^{\frac{3}{2}}}+\frac{9\cc z}{(x^2+y^2+z^2)^{\frac{5}{2}}}
-\frac{15\cc z^3}{(x^2+y^2+z^2)^{\frac{7}{2}}}\\
& =-z-\frac{z}{r^3}+\frac{9\cc z}{r^5}-\frac{15\cc z^3}{r^7},
\end{split}
\end{equation}
where $r=(x^2+y^2+z^2)^{\frac{1}{2}}$.

To find the equilibrium points we have to solve the system
\[
    \left.
                \begin{array}{ll}
                  \Omega_{x} = 0\\
                  \Omega_{y} = 0\\
                  \Omega_{z} = 0
                \end{array}
              \right\}\Rightarrow
    \left.
                \begin{array}{ll}
                 \displaystyle \left(\lambda_2-\frac{1}{r^3}+\frac{3\cc }{r^5}-\frac{15\cc z^2}{r^7}\right)x :=Ax =0\\[0.5em]
                 \displaystyle\left(\lambda_1-\frac{1}{r^3}+\frac{3\cc }{r^5}-\frac{15\cc z^2}{r^7}\right)y :=B y= 0\\[0.5em]
                 \displaystyle \left(-1-\frac{1}{r^3}+\frac{9\cc }{r^5}-\frac{15\cc z^2}{r^7}\right)z:=C z = 0
                \end{array}
              \right\}
\]

In the above expressions $A$ and $B$ cannot simultaneously equal to $0$ since $\lambda_1\neq\lambda_2$.
Also, $A$ and $C$, or $B$ and $C$, cannot simultaneously equal to $0$, since $A-C=\lambda_2+1-\frac{6\cc }{r^5}>0$ because $\cc <0$; a similar argument holds for $B$ and $C$. This implies that, for example, if $A= 0$, then $B\neq 0$ and $C\neq 0$, so $y=z=0$ and $x$ is given by the equation $A=0$; the same reasoning applies for the other combinations of variables.
Thus, all equilibrium points must lie on the $x$-, $y$-, $z$-coordinate axes.
Precisely, we have the following results.

\begin{description}

\item[$i)$ Equilibrium points on the $x$-axis] In the case  $A= 0$, $B\neq0$, $C\neq0$,
we must have $y=z=0$. From $A=0$ and $z=0$ we infer $\displaystyle h_A(r):=\lambda_2-\frac{1}{r^3}+\frac{3\cc }{r^5}=0$.
We have $\displaystyle  h'_A(r)= \frac{3}{r^4}-\frac{15\cc }{r^6}>0$, since $\cc <0$; also, $\lim_{r\to 0}h_A(r)=-\infty$
and  $\lim_{r\to \infty}h_A(r)=\lambda_2>0$. Hence, the equation $h_A(r)=0$ has a unique solution $r^*_x>0$,
yielding the equilibrium points $(\pm r^*_x,0,0)$.

\item[$ii)$ Equilibrium points on the $y$-axis] In the case  $B= 0$, $A\neq0$, $C\neq0$, we must have $x=z=0$.
From $B=0$ and $z=0$ we infer $\displaystyle h_B(r):=\lambda_1-\frac{1}{r^3}+\frac{3\cc }{r^5}=0$.
We have $\displaystyle  h'_B(r)= \frac{3}{r^4}-\frac{15\cc }{r^6}>0$, since $\cc <0$; also,
$\lim_{r\to 0}h_B(r)=-\infty$ and  $\lim_{r\to \infty}h_B(r)=\lambda_1>0$. Hence,
the equation $h_B(r)=0$ has a unique solution $r^*_y>0$, yielding the equilibrium points $(0,\pm r^*_y,0,0)$.

\item[$iii)$ Equilibrium points on the $z$-axis] In the case  $C= 0$, $A\neq0$, $B\neq0$, we must have $x=y=0$, so $z=\pm r$.
Hence $C=0$ implies $-1-\frac{1}{r^3}+\frac{9\cc }{r^5}-\frac{15\cc r^2}{r^7}=-1-\frac{1}{r^3}-\frac{6\cc }{r^5}
=(-r^5-r^2-6\cc )/r^5=0$.
Let $h_C(r)=-r^5-r^2-6\cc $. We have $h'_C(r)=-5r^4-2r<0$; also, $\lim_{r\to 0} h_C(r)=-6\cc >0$ and
$\lim_{r\to +\infty} h_C(r)=-\infty$. Hence, the equation $h_C(r)=0$ has a unique solution $r^*_z>0$, yielding
the equilibrium points $(0,0,\pm r^*_z)$.
\end{description}
In the case of the Sun-Jupiter-Hektor system, in normalized units,
we obtain $\lambda_1= 0.0021444999866622183$, $\lambda_2=2.997855500013338$, and the equilibrium points location are
given by the following figures.
\[
\begin{tabular}{|l|l|l|l|}
                              \hline
                               & $x$ & $y$ & $z$ \\
                              \hline
                              $x$-axis equilibria & $\pm 0.6935267570$  & 0 & 0 \\
                              $y$-axis equilibria & 0 & $\pm 7.7545747196$  & 0 \\
                              $z$-axis equilibria & 0 & 0 & $\pm 0.0008923544$ \\
                              \hline
                            \end{tabular}
\]

We remark that in the case of the Hill's four body problem without a non-oblate tertiary, the $x$-axis equilibria
and the $y$-axis equilibria also exist, see \cite{Burgos_Gidea}; their locations, in the case of Hektor,
are very close to the ones in the case of an oblate tertiary. Precisely, we have the following results.

\[
\begin{tabular}{|l|l|l|l|}
 \hline
 & $x$ & $y$ & $z$ \\
\hline
$x$-axis equilibria & $\pm 0.6935265657$  & 0 & 0 \\
$y$-axis equilibria & 0 & $\pm 7.7545747024$  & 0 \\
\hline
\end{tabular}
\]

This result leads us to conclude that the $x$-axis equilibria and the $y$-axis equilibria for the
Hill's problem with oblate tertiary are continuations of the ones for the Hill's problem
with non-oblate tertiary. On the other hand, the $z$-axis equilibria do not exist for
the Hill's problem with non-oblate tertiary, so they are new features of the Hill's problem with oblate tertiary.

To summarize, the Hill's three-body problem has 2 equilibrium points, Hill's four-body problem has 4 equilibrium points,
and the Hill's four-body problem with oblate tertiary has 6 equilibrium points.

\vskip.1in

In Fig. \ref{fig:Hektor_z_dependence_on_c} we plot the dependence on the $\cc $ of the distance from $z$-axis
equilibrium point to the origin (in km), when we let the parameter $C_{20}$ range between -0.001 and -0.95.
The estimates on the axes of Hektor are $a=208$ km, $b=65.5$ km, $c=60$  km (\cite{DESCAMPS2015}).
Using the ellipsoid model, we find $C_{20}=-0.476775$ and hence $z\simeq 100$ km, which lies
outside the asteroid. Using $C_{20}=-0.15$ as provided by \cite{Marchis},
we obtain $x\simeq 62$ km, which basically coincides with the surface of the asteroid.
We notice that the vertical equilibrium is not outside the Brillouin sphere, which
corresponds to the region where the spherical harmonic series expansion is convergent.
Inside the Brillouin sphere the series is divergent, if the shape is an ellipsoid.
Given that the true shape of the asteroid is unknown, we cannot determine precisely
the region where the spherical harmonic series is convergent or divergent and, hence, we cannot
decide whether the $z$-equilibrium is indeed real, but just postulate its existence.

\begin{figure}
\includegraphics[width=0.85\textwidth]{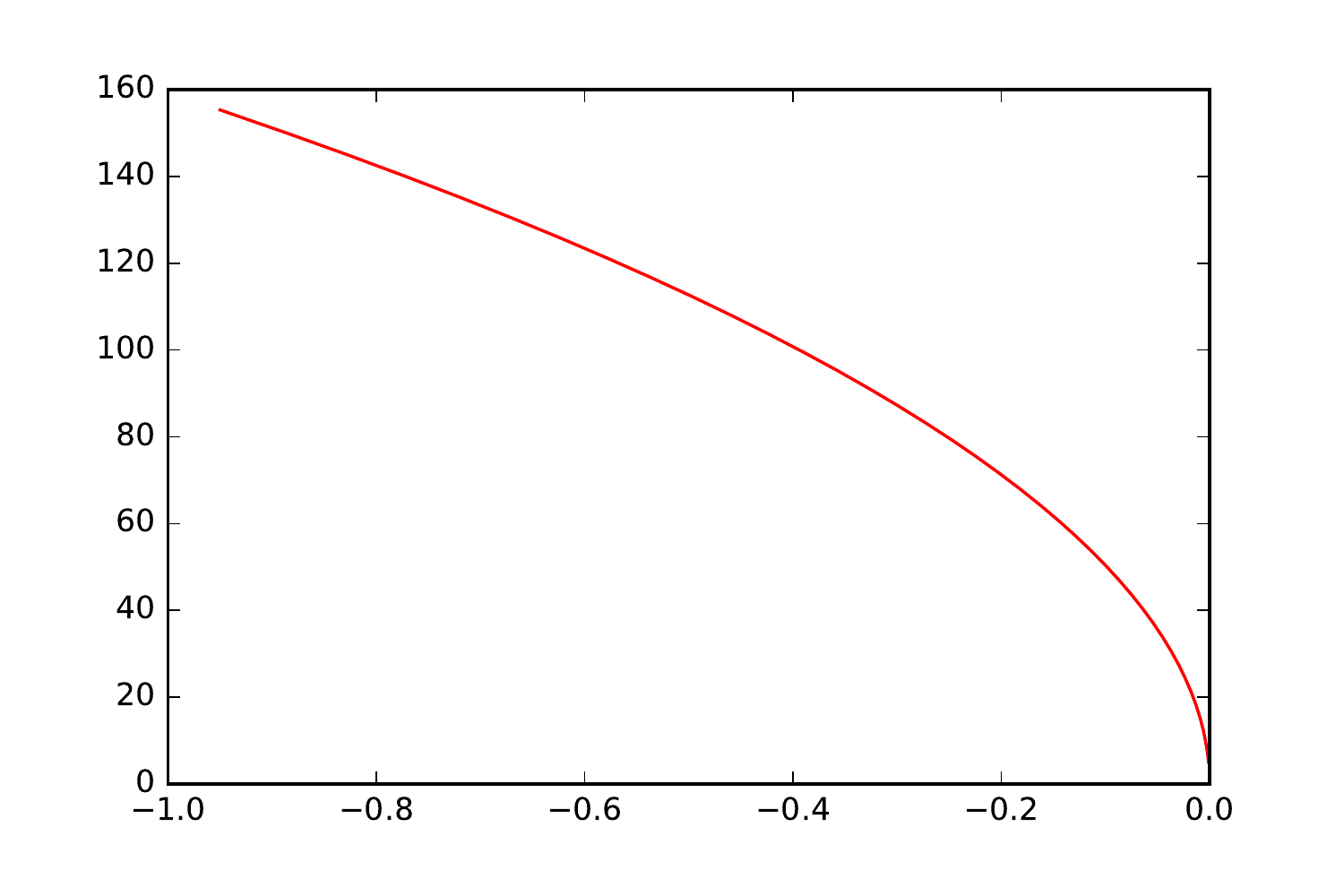}
\caption{The dependence of the $z$-axis equilibrium point distance on the $C_{20}$
rescaled, see \eqref{eqn:normalized_\cc }.}
\label{fig:Hektor_z_dependence_on_c}
\end{figure}

To convert to real units, the distances from the equilibrium points to the center need to be multiplied by $m_3^{1/3}$ -- due to the rescaling involved in the Hill procedure --, and by the unit of distance which in this case is the distance Sun-Jupiter. It follows that the $x$-axis equilibrium points are at a distance of $85,512.774$ km from Hektor,
the $x$-axis equilibrium points are at a distance of $956,149.406$ km, and the
 $z$-axis equilibrium points are at a distance of $110.028$ km. As the smallest semi-minor axis of Hektor is
$60$ km, the $z$-axis equilibrium points are outside the body of the asteroid. It seems though that for many other asteroids the $z$-axis equilibrium points are located inside their body.

\begin{rem}
These  $z$-axis equilibria also appear in the case of the motion
of a particle in a geopotential field (see, e.g.,
\cite{celletti2014dynamics}). From that model it can be derived
that the distance from the $z$-axis equilibrium points to the
center is given by  $\hat r_z= R_3(-3C_{20})^{1/2}$. When we apply
this formula in the case of Hektor, the numerical result is very
close to the one found above. This formula can be derived from the
equation  $h_C(r)=0$ if we drop the term $r^5$.
\end{rem}

\subsection{Linear stability of the equilibrium points}
We study the linear stability of the equilibrium points in the case of Hektor.

The Hamiltonian \eqref{eqn:hill_rotated} yields the following system of equations
\begin{eqnarray*}
&\dot x=v_x, &\dot {v}_x=2v_y+ \Omega_x,\\
&\dot y=v_y, &\dot {v}_y=-v_x+ \Omega_y,\\
&\dot z=v_z,    &\dot{v}_z= \Omega_z,
\end{eqnarray*}
where $\Omega$ is the effective potential given by \eqref{eqn:eff_poten_rot}
(again, we omit the overline bar on the variables).

The second order derivatives of  $\Omega$ are given by
\begin{equation}\begin{split}
\Omega_{xx}=&\lambda_2-\frac{1}{r^3}+\frac{3x^2}{r^5}+\frac{3\cc}{r^5}
-\frac{15\cc x^2}{r^7}-\frac{15\cc z^2}{r^7}+\frac{105\cc z^2x^2}{r^9},\\
\Omega_{yy}=&\lambda_1-\frac{1}{r^3}+\frac{3y^2}{r^5}+\frac{3\cc}{r^5}
-\frac{15\cc y^2}{r^7}-\frac{15\cc z^2}{r^7}+\frac{105\cc z^2y^2}{r^9},\\
\Omega_{zz}=&-1-\frac{1}{r^3}+\frac{3z^2}{r^5}+\frac{9\cc}{r^5}-\frac{90\cc z^2}{r^7}+\frac{105\cc z^4}{r^9},\\
\Omega_{xy}=&\frac{3xy}{r^5}-\frac{15\cc xy}{r^7}+\frac{105\cc z^2xy}{r^9},\\
\Omega_{xz}=&\frac{3xz}{r^5}-\frac{45\cc xz}{r^7}+\frac{105\cc z^3x}{r^9},\\
\Omega_{yz}=&\frac{3yz}{r^5}-\frac{45\cc yz}{r^7}+\frac{105\cc z^3y}{r^9}.
\end{split}\end{equation}

The Jacobian matrix describing the linearized system is
\begin{equation}\label{eqn:jacobi}
\mathscr{J}=\left(
               \begin{array}{rrrrrr}
                  0 & 0 & 0 & 1 & 0 & 0 \\
                 0 & 0 & 0 & 0 & 1 & 0 \\
                  0 & 0 & 0 & 0 & 0 & 1 \\
                 \Omega_{xx} & \Omega_{xy} & \Omega_{xz} & 0 & 2 & 0 \\
                 \Omega_{yx} & \Omega_{yy} & \Omega_{yz} &-2 & 0 & 0 \\
                 \Omega_{zx} & \Omega_{zy} & \Omega_{zz} & 0 & 0 & 0 \\
               \end{array}
             \right).
\end{equation}
Since the equilibria are of the form $(\pm r^*_x,0,0)$, $(0,\pm r^*_y,0)$, $(0,0,\pm r^*_z)$, the mixed second order partial  derivatives $\Omega_{xy}$, $\Omega_{xz}$, $\Omega_{yz}$ vanish at each of the equilibrium points.

Hence Jacobian matrix \eqref{eqn:jacobi} evaluated at the equilibria is of the form:
\begin{equation}\label{eqn:jacobi_at_z}
\mathscr{J}=\left(
               \begin{array}{rrrrrr}
                  0 & 0 & 0 & 1 & 0 & 0 \\
                 0 & 0 & 0 & 0 & 1 & 0 \\
                  0 & 0 & 0 & 0 & 0 & 1 \\
                 \Omega_{xx} &0 & 0 & 0 & 2 & 0 \\
                0 & \Omega_{yy} & 0 &-2 & 0 & 0 \\
                0 & 0 & \Omega_{zz} & 0 & 0 & 0 \\
               \end{array}
             \right),
\end{equation}

The characteristic equation of \eqref{eqn:jacobi_at_z} is
\begin{equation}\label{eqn:charcteristic}
    (\rho^2-\Omega_{zz})(\rho^4+(4-\Omega_{xx}-\Omega_{yy})\rho^2+\Omega_{xx}\Omega_{yy})=0.
\end{equation}

The stability of the equilibria  depends on the signs of $\Omega_{zz}$ and of $A$, $B$, and $D$.

We find the following stability character of the equilibrium positions in the case of the Sun-Jupiter-Hektor system:
\begin{description}
\item[$i)$ Eigenvalues of $x$-axis equilibria at $(\pm 0.6935267570,  0 ,0)$]
\begin{eqnarray*} &2.50694248&-2.50694248,\\
        &2.07048307i,  &-2.07048307i,\\
        &1.99946504i,  &- 1.99946504i.
\end{eqnarray*}
\textbf{Stability type:} center $\times $ center $\times $ saddle.

\item[$ii)$ Eigenvalues of $y$-axis equilibria at $(0,7.7545747196,0)$]
\begin{eqnarray*}
& 0.98901573i,   &-0.98901573i,\\
& 0.14036874i,   &-0.14036874i,\\
& 1.00107168i,   &-1.00107168i
\end{eqnarray*}
\textbf{Stability type:} center $\times $ center $\times $ center.

\item[$iii)$ Eigenvalues of $z$-axis equilibria at $(0,0,\pm 0.0008923544)$]
\begin{eqnarray*}
&37514.0432165187+0.9999999998i, &-37514.0432165187+0.9999999998j,\\
&37514.0432165187-0.9999999998i, &-37514.0432165187-0.9999999998i\\
&53052.8687i,                    &-53052.8687i
\end{eqnarray*}
\textbf{Stability type:} center $\times $ complex saddle.
\end{description}

Note that for the $z$-axis equilibria the imaginary part of the `Krein quartet' of
eigenvalues  of $z$-axis equilibria is approximately $\pm 1$.
This means that the infinitesimal motion around the equilibrium point is close to the $1:1$ resonance with the rotation of the primary and the secondary.
In Fig. \ref{fig:Hektor_eigenvalues} we show that for a range of $r^*_z$ values between $z=0.000892354498497342$ (corresponding to the $\cc $ value for Hektor) and $z=0.01$ (corresponding to $\cc =-1.666668333\times 10^{-5}$), the real part and the imaginary part of the `Krein quartet' of
eigenvalues; the imaginary part stays close to $\pm 1$.
\begin{figure}$\begin{array}{cc}
\includegraphics[width=0.5\textwidth]{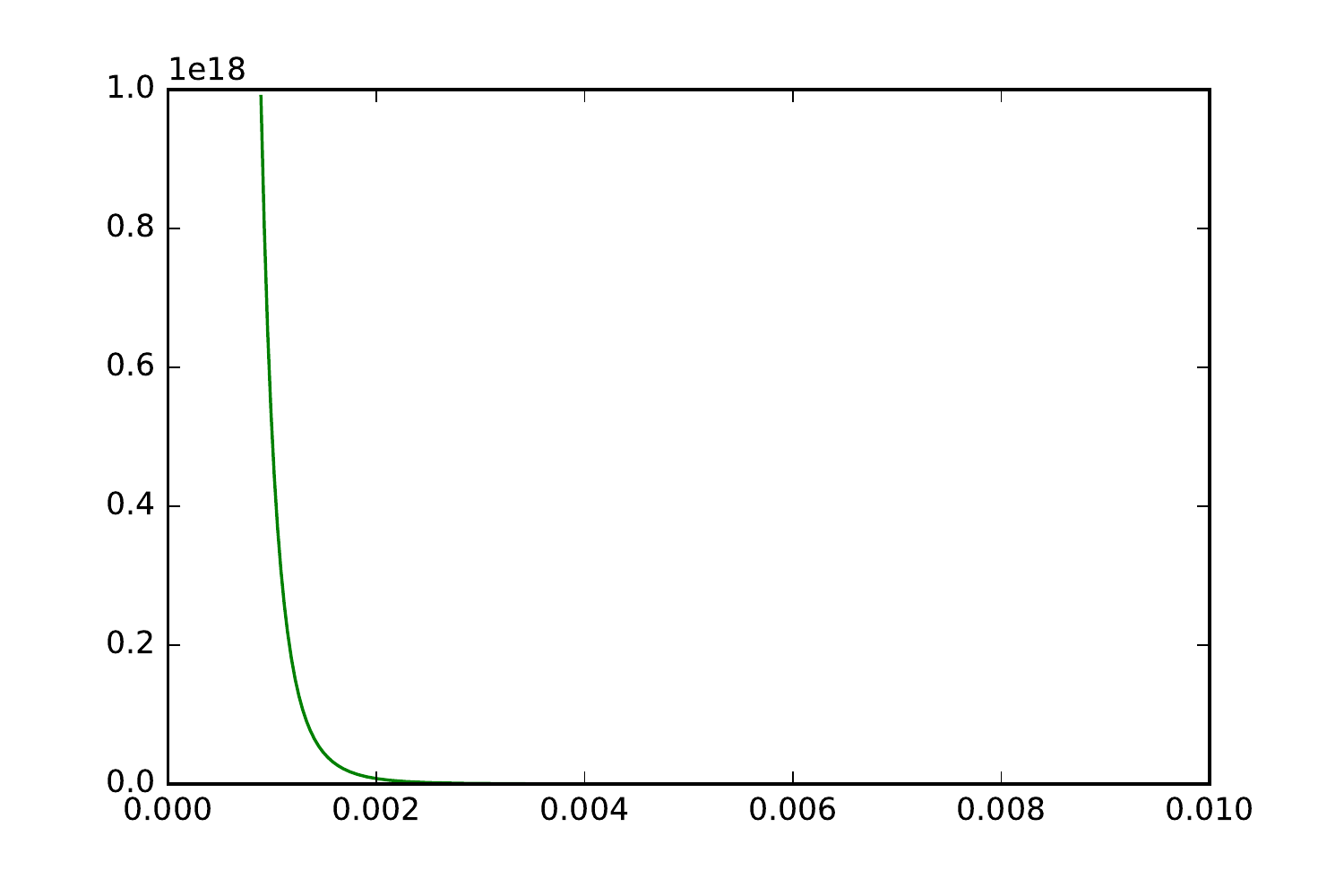} &
\includegraphics[width=0.5\textwidth]{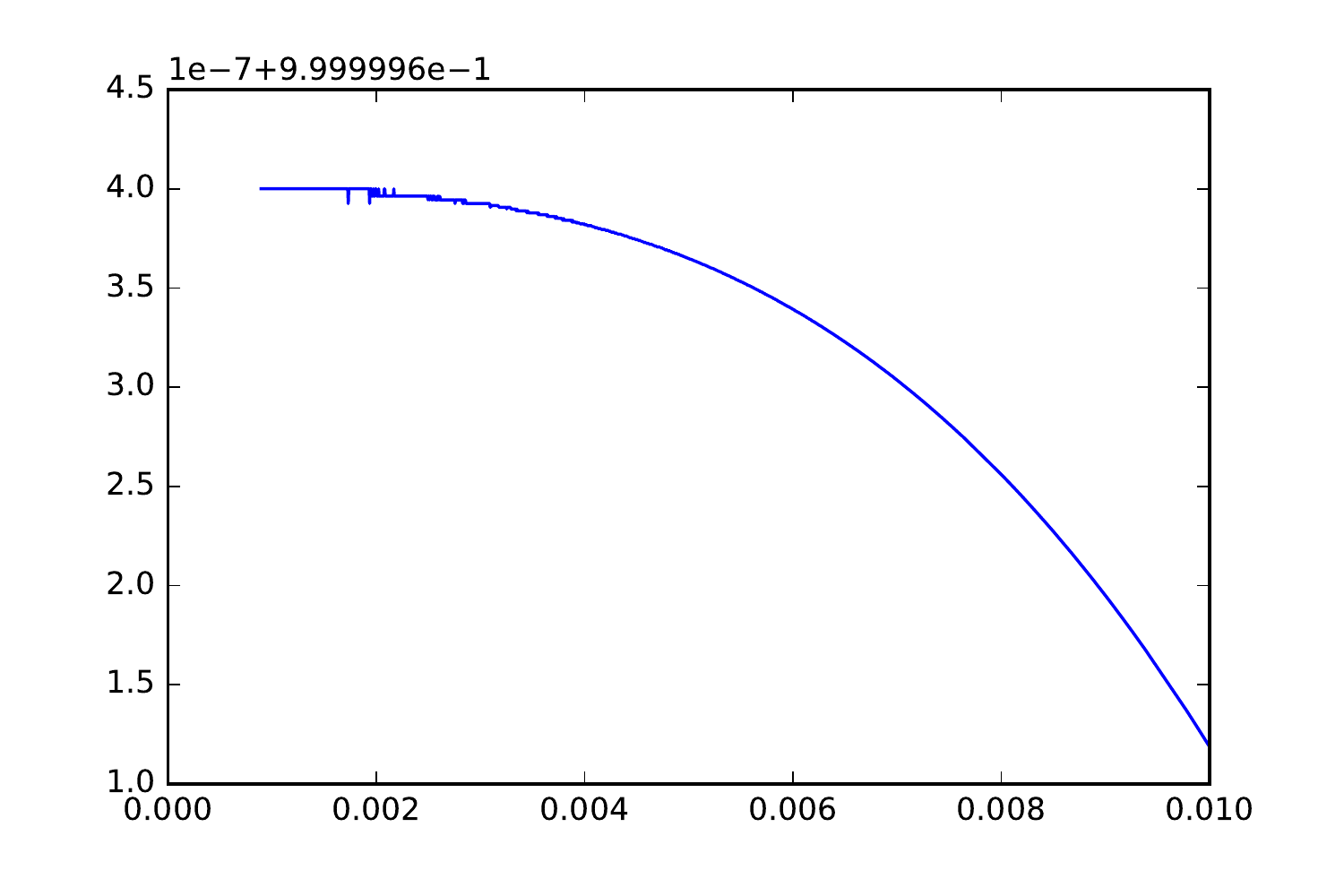}
\end{array}$
\caption{The dependence of the real part (left) and imaginary part (right) of the  Krein quartet of
eigenvalues on  the $z$-axis equilibrium point. The horizontal axis represents the distance $r^*_z$  from the equilibrium point to the origin,
the vertical axis the real part (left), and the absolute value of the imaginary part (right) of the eigenvalues. The former never changes sign, and the latter stays within $4\times 10^{-7}$ from $1$.}
\label{fig:Hektor_eigenvalues}
\end{figure}

In Section \ref{sec:z_stability} we will provide an analytic argument that the real part of the `Krein quartet' of
eigenvalues is always non-zero, and the imaginary part is  close to $\pm 1$  for $r^*_z$ sufficiently small;
this result will help us to explain the behavior observed in Fig.~\ref{fig:Hektor_eigenvalues}.

In the sequel we give a more detailed analysis of the linear stability of all equilibria, for a wide range of parameters $\mu$ and $\cc$.

\subsubsection{Linear stability of the equilibria on the $z$-axis}
\label{sec:z_stability}

The $z$-axis equilibrium points are of the form $(0,0,\pm \rz)$, with
\begin{equation}\label{eqn:z_eq}
-(\rz)^5-(\rz)^2-6\cc =0,
\end{equation}
which yields\begin{equation}\label{eqn:z_c20} \cc=\frac{-(\rz)^2- (\rz)^5}{6}.
\end{equation}
Evaluating $\Omega_{xx}$, $\Omega_{yy}$, $\Omega_{zz}$ at the equilibrium point
yields:

\begin{equation*}\begin{split}
\Omega_{xx}=& \lambda_2 - (r^*_z)^{-3}-12\cc (r^*_z)^{-5},\\
\Omega_{yy}=& \lambda_1 - (r^*_z)^{-3}-12\cc (r^*_z)^{-5},\\
\Omega_{zz}=& -1+2(r^*_z)^{-3}+24\cc (r^*_z)^{-5}.
\end{split}\end{equation*}
Substituting in \eqref{eqn:z_c20}
we have
\begin{equation}\label{eqn:Omega_e}
\begin{split}
\Omega_{xx}&=2+\lambda_2+(r^*_z)^{-3},\\
\Omega_{yy}&=2+\lambda_1+(r^*_z)^{-3},\\
\Omega_{zz}&=-5-2(r^*_z)^{-3}.
\end{split}\end{equation}
Using \eqref{eq18} and denoting $d:=\sqrt{1-v^2(4-v^2)(\mu-\mu^2)}$ we can write
\begin{equation}\label{eqn:lambda}
\begin{split}
\lambda_1=\frac{3}{2}(1-d),\\
\lambda_2=\frac{3}{2}(1+d).
\end{split}
\end{equation}

Also for $\cc=0$   we have $d_0=\sqrt{1-3(\mu-\mu^2)}$ and
\begin{equation}\label{eqn:lambda0}
\begin{split}
\lambda_{10}=&\frac{3}{2}(1-d_0),\\
\lambda_{20}=&\frac{3}{2}(1+d_0).
\end{split}
\end{equation}
This is in agreement with the results in  \cite{Burgos_Gidea}.

For future reference, we expand $d$ as a power series in the parameter $\cc$ as
\begin{equation}\label{eqn:d_power}
d=d_0+d_1 \cc+ O(\cc^2),
\end{equation}
where the coefficient  $d_1$   can be obtained can be obtained from
the Taylor's theorem around $\cc = 0$ as
\begin{equation}\label{eqn:d_coeff}
\begin{split}
d_1=&-\frac{2(\mu-\mu^2)}{d_0} m_3^{2/3}.
\end{split}
\end{equation}

From the characteristic  equation \eqref{eqn:charcteristic},  we obtain that the pair of  eigenvalues   $\rho_{1,2}=\pm (\Omega_{zz})^{1/2}$ is purely imaginary, since by \eqref{eqn:Omega_e}, $\Omega_{zz}<0$.

The `Krein quartet' eigenvalues  are  given by
\begin{equation}\label{eqn:quadratic}
    \rho_{3,4,5,6}= \pm\sqrt{\frac{-A\pm\sqrt{A^2-4B}}{2}},
\end{equation}
where
\begin{equation*}\begin{array}{lllll}A&=&\displaystyle 4-\Omega_{xx}-\Omega_{yy}&=&-3-\frac{2}{(r^*_z)^3},\\
B&=&\Omega_{xx}\Omega_{yy}&=&\displaystyle 10+{9\over 4} v^2(4-v^2)(\mu-\mu^2)
+\frac{7}{(r^*_z)^3}+\frac{1}{(r^*_z)^6}.\end{array}\end{equation*}

Then we have
\begin{equation*}\begin{split}
D:=A^2-4B=&d^2-40-\frac{16}{(r^*_z)^3}=-31-9v^2(4-v^2)(\mu-\mu^2)-\frac{16}{(r^*_z)^3}<0.
\end{split}\end{equation*}

Since $-A>0$ and $D<0$, we obtain that the eigenvalues $\rho_{3,4,5,6}$ are complex numbers, non-real, non-purely-imaginary, for all parameter values. Let $\rho=a+ib$ be such that $\rho^2=-\frac{A}{2}\pm\frac{\sqrt{4B-A^2}}{2}i:=\alpha+i\beta$. We have  \[a+ib=\left(\frac{(\alpha^2+\beta^2)^{\frac{1}{2}}+\alpha}{2}\right)^{\frac{1}{2}}+
\textrm{sign}(\beta)\left(\frac{(\alpha^2+\beta^2)^{\frac{1}{2}}-\alpha}{2}\right)^{\frac{1}{2}}\ i.
\]

To show that $b$ is approximately $\pm 1$, or $b^2\approx 1$, for $r^*_z\approx 0$, note that
\begin{equation*}\begin{split}
b^2=&\frac{(\alpha^2+\beta^2)^{\frac{1}{2}}-\alpha}{2}=\frac{A}{4}+\frac{\sqrt{B}}{2}\\
=&-\frac{3}{4}+\frac{1}{2}\left[\left(10+\frac{9}{4}\Upsilon+\frac{7}{(r^*_z)^3}+
\frac{1}{(r^*_z)^6} \right)^{\frac{1}{2}}-\frac{1}{(r^*_z)^3}\right]\\
=&-\frac{3}{4}+\frac{1}{2}\frac{10+\frac{9}{4}\Upsilon+\frac{7}{(r^*_z)^3}+
\frac{1}{(r^*_z)^6} -\frac{1}{(r^*_z)^6}}{\left(10+\frac{9}{4}\Upsilon+\frac{7}{(r^*_z)^3}+
\frac{1}{(r^*_z)^6} \right)^{\frac{1}{2}}+\frac{1}{(r^*_z)^3}}\\
=&-\frac{3}{4}+\frac{1}{2}\frac{10+\frac{9}{4}\Upsilon+\frac{7}{(r^*_z)^3}}{\left(10+\frac{9}{4}\Upsilon+\frac{7}{(r^*_z)^3}+
\frac{1}{(r^*_z)^6} \right)^{\frac{1}{2}}+\frac{1}{(r^*_z)^3}},
\end{split}
\end{equation*}
where $\Upsilon:=v^2(4-v^2)(\mu-\mu^2)$.
Since
\[\lim_{r^*_z\to 0} \frac{10+\frac{9}{4}\Upsilon+\frac{7}{(r^*_z)^3}}{\left(10+\frac{9}{4}\Upsilon+\frac{7}{(r^*_z)^3}+
\frac{1}{(r^*_z)^6} \right)^{\frac{1}{2}}+\frac{1}{(r^*_z)^3}}  = \frac{7}{2},\]
we have that $\lim _{r^*_z\to 0} b^2=-\frac{3}{4}+\frac{7}{4}=1$, so $b^2\approx 1$ for $r^*_z\approx 0$, as in the case of Hektor.

We obtained  the following result:
\begin{prop}\label{prop:_lin_stab}
Consider the equilibria on the $z$-axis. For   $\mu\in(0,1/2]$, $\Omega_{zz}$, $A$ and $D$ are negative. Consequently, one pair of eigenvalues is purely imaginary, and the two other pairs of eigenvalues are complex conjugate, with the imaginary part close to  $\pm i$ for $\cc$ negative and sufficiently small. The linear stability is of  center $\times$ complex-saddle type.
\end{prop}

\subsubsection{Linear stability of the equilibria on the $y$-axis}
\label{sec:y_stability}

The $y$-axis equilibrium points are of the form $(0,\pm r^*_y,0)$, with
\begin{equation}\label{eqn:y_eq}\lambda_1 (\ry)^5-(\ry)^2+3\cc=0,\end{equation}
which yields
\begin{equation}\label{eqn:y_c20}
\cc=\frac{(\ry)^2-\lambda_1 (\ry)^5}{3}.
\end{equation}
Evaluating $\Omega_{xx}$, $\Omega_{yy}$, $\Omega_{zz}$ at the equilibrium point
yields:
\begin{equation}\label{eqn:y_Omega_second_der_0}
\begin{split}
\Omega_{xx}=&\lambda_2-\frac{1}{(\ry)^3}+\frac{3\cc}{(\ry)^5},\\
\Omega_{yy}=&\lambda_1+\frac{2}{(\ry)^3}-\frac{12\cc}{(\ry)^5},  \\
\Omega_{zz}=& -1-\frac{1}{(\ry)^3}+\frac{9\cc}{(\ry)^5}.
\end{split}\end{equation}

Substituting $\cc$ from \eqref{eqn:y_c20} we obtain
\begin{equation}\label{eqn:y_Omega_second_der}\begin{split}
\Omega_{xx}=&\lambda_2-\lambda_1 ,\\
\Omega_{yy}=& 5\lambda_1-\frac{2}{(\ry)^3},  \\
\Omega_{zz}=& -1-3\lambda_1+\frac{2}{(\ry)^3},
\end{split}\end{equation}

We also expand $\ry$ as a power series in the parameter $\cc$ as
\begin{equation}\label{eqn:y_r_power}
\ry=\ryo+r_{y1} \cc+O(\cc^2),
\end{equation}
where $\pm\ryo$ is the position of the $y$-equilibrium in the case when $\cc=0$, which is given by
$\ryo^3=1/\lambda_{10}$; this is in agreement with \cite{Burgos_Gidea}.
The computation of $r_{y1}$   yields
\begin{equation}\label{eqn:y_r_coeff}
\begin{split}
r_{y1}=& \frac{-1 +(1/2)d_1\ryo^5}{\ryo},
\end{split}
\end{equation}
with $d_1$ as in \eqref{eqn:d_coeff}.

We will also need $\frac{1}{(\ry)^3}$ as a power series in the parameter $\cc$
\begin{equation}\label{eqn:y_r_cube_inv_power}
\frac{1}{(\ry)^3}=\alpha+\beta \cc+O(\cc^2),
\end{equation}
and a simple calculation yields
\begin{equation}\label{eqn:y_r_cube_inv}
\begin{split}
\alpha=&\frac{1}{\ryo^3},\\
\beta=&  -\frac{3r_{y1}}{\ryo^4}.
\end{split}
\end{equation}

For $\mu=1/2$, we have $d_0=\frac{1}{2}$, $\lambda_{10}=\frac{3}{4}$, $d_1=-m_3^{2/3}$, $r_{y0}=\left(\frac{4}{3}\right)^{1/3}$.
It is easy to see that  dominant part $d_0$ of $d$ is a strictly decreasing function with respect to $\mu\in(0,1/2]$ and takes values in $[1/2,1)$.
The dominant part $\lambda_{10}$ of $\lambda_1$ is increasing with respect to $\mu\in(0,1/2]$ and takes values in $(0,3/4] $.
Also, the dominant part $\ryo$ of $\ry$ is a strictly decreasing function in $\mu\in(0,1/2]$, where   $\ryo(1/2)=\sqrt[3]{4/3}$ and $\ryo\rightarrow\infty$ when $\mu\rightarrow 0$; as a consequence the values of $\ryo$ are in the interval $[\sqrt[3]{4/3},\infty)$.

From \eqref{eqn:y_Omega_second_der_0} we have
\begin{equation*}\begin{split}
\Omega_{zz} & =-1-\frac{1}{(\ry)^3}+\frac{9\cc}{(\ry)^5}\\
& =  -\frac{1}{(\ry)^5}((\ry)^5+(\ry)^2-9\cc)\\
& < 0
\end{split}\end{equation*}
since $\ry>0$ and $\cc$ is negative.
Therefore, $\Omega_{zz}<0$ for all admissible values of $\mu$.

For $A=4-\Omega_{xx}-\Omega_{yy}$, using \eqref{eqn:lambda0} and the expansions \eqref{eqn:d_power} and \eqref{eqn:y_r_cube_inv_power} we obtain
\begin{equation*}\begin{split}
A   & = 1-3\lambda_1 +\frac{2}{(\ry)^3}\\
    & = 1-3\lambda_{10}+\frac{2}{(\ryo)^3}+O(\cc)\\
    & = 1-3\lambda_{10}+2\lambda_{10}+O(\cc)\\
    & > 0
\end{split}\end{equation*}
for $\cc$ small.

For $B=\Omega_{xx}\Omega_{yy}$ using \eqref{eqn:lambda0} and  the expansions \eqref{eqn:d_power} and \eqref{eqn:y_r_cube_inv_power} we obtain
\begin{equation*}\begin{split}
 B  & = (\lambda_2-\lambda_1)\left (5\lambda_1 -\frac{2}{(\ry)^3}\right)\\
    & = (3d)\left (\frac{15}{2} -\frac{15d}{2} -\frac{2}{(\ry)^3}\right)\\
    & = (3d_0)\left(5\lambda_{10}-\frac{2}{\ryo^3}\right)+O(\cc)\\
    & = (3d_0)\left(5\lambda_{10}-2\lambda_{10}\right)+O(\cc)\\
    & > 0
\end{split}\end{equation*}
for $\cc$ small.

For $D=A^2-4B$, using \eqref{eqn:lambda0} and  the expansions \eqref{eqn:d_power} and \eqref{eqn:y_r_cube_inv_power} we have
\begin{equation*}\begin{split}
 D & =\left(1-3\lambda_{10}+\frac{2}{\ryo^3}\right)^2-4(3d_0)\left(5\lambda_{10}-\frac{2}{\ryo^3}\right)+O(\cc)\\
   & =\left(1-\lambda_{10}\right)^2-12(3-2\lambda_{10})\lambda_{10}+O(\cc).
\end{split}\end{equation*}
For $\mu\approx 0$ we have $D\approx 1+O(\cc)$ and for $\mu=1/2$ we have $D=-\frac{215}{16}+O(\cc)$.
The intermediate value theorem implies that $D$ changes its sign from positive to negative  for $\mu\in(0,1/2]$, provided  $\cc$ is small.
 We have thus proved the following result:

\begin{prop}\label{prop:y_lin_stab}
Consider the equilibria on the $y$-axis. For   $\mu\in(0,1/2]$  and for the parameter $\cc$ negative and small enough, $\Omega_{zz}$ is always negative,  the coefficients $A$ and $B$ are always positive, and the value of the discriminant $D$ changes from positive to negative values. Consequently, one pair of eigenvalues is always purely imaginary, and  there exists   $\mu_*$, depending on $c_{20}$, where the other two pairs of  eigenvalues change from being purely imaginary to  being complex conjugate. The linear stability changes from   center $\times$ center $\times$ center type to  center $\times$ complex-saddle type.
\end{prop}

\subsubsection{Linear stability of the equilibria on the $x$-axis}

The $x$-axis equilibrium points are of the form $(\pm \rx, 0,0)$, with
\begin{equation}\label{eqn:x_eq}\lambda_2(\rx)^5-(\rx)^2+3\cc=0,\end{equation}
which yields
\begin{equation}\label{eqn:x_c20}
\cc=\frac{(\rx)^2-\lambda_2 (\rx)^5}{3}.
\end{equation}

Evaluating $\Omega_{xx}$, $\Omega_{yy}$, $\Omega_{zz}$  at the equilibrium point
yields:
\begin{equation}\label{eqn:x_Omega_second_der_0}
\begin{split}
\Omega_{xx}=&\lambda_2+\frac{2}{(\rx)^3}-\frac{12\cc}{(\rx)^5},\\
\Omega_{yy}=&\lambda_1-\frac{1}{(\rx)^3}+\frac{3\cc}{(\rx)^5},  \\
\Omega_{zz}=& -1-\frac{1}{(\rx)^3}+\frac{9\cc}{(\rx)^5}.
\end{split}\end{equation}

Substituting $\cc$ from \eqref{eqn:x_c20} we obtain
\begin{equation}\label{eqn:x_Omega_second_der}\begin{split}
\Omega_{xx}=&5\lambda_2-\frac{2}{(\rx)^3} ,\\
\Omega_{yy}=& \lambda_1-\lambda_2,  \\
\Omega_{zz}=& -1-3\lambda_2+\frac{2}{(\rx)^3},
\end{split}\end{equation}

We  expand $\rx$ as a power series in the parameter $\cc$ as
\begin{equation}\label{eqn:r_power}
\rx=\rxo+r_{x1} \cc+ O(\cc^2),
\end{equation}
where $\pm\rxo$ is the position of the $x$-equilibrium in the case when $\cc=0$, which is given by
$\rxo^3=1/\lambda_{20}$; see \cite{Burgos_Gidea}.
The computation of $r_{x1}$  yields
\begin{equation}\label{eqn:x_r_coeff}
\begin{split}
r_{x1}=& \frac{-1-(1/2)d_1\rxo^5}{\rxo}.
\end{split}
\end{equation}

We will also need $\frac{1}{(\rx)^3}$ as a power series in the parameter $\cc$
\begin{equation}\label{eqn:x_r_cube_inv_power}
\frac{1}{(\rx)^3}=\alpha'+\beta' \cc+O(\cc^2),
\end{equation}
and a simple calculation yields
\begin{equation}\label{eqn:x_r_cube_inv}
\begin{split}
\alpha' = & \frac{1}{\rxo^3},\\
\beta'  = &  -\frac{3r_{x1}}{\rxo^4}.
\end{split}
\end{equation}

From \eqref{eqn:x_Omega_second_der_0} we have
\begin{equation*}\begin{split}
\Omega_{zz} & = -1-\frac{1}{(\rx)^3}+\frac{9\cc}{(\rx)^5}\\
&=-\frac{1}{(\rx)^5}((\rx)^5+(\rx)^2-9\cc)\\
&<0
\end{split}\end{equation*}
since $\rx>0$ and $\cc<0$.
Therefore, $\Omega_{zz}<0$ for all admissible values of $\mu$.

For $A=4-\Omega_{xx}-\Omega_{yy}$, using \eqref{eqn:lambda0} and   the expansions \eqref{eqn:d_power} and \eqref{eqn:x_r_cube_inv_power}, we obtain
\begin{equation*}\begin{split}
 A & = 1-3\lambda_{20}+\frac{2}{(\rxo)^3}+O(\cc)\\
  & = 1-\lambda_{20}+O(\cc)\\
  &=-\frac{1}{2}-\frac{3}{2}d_0+O(\cc)\\
  & < 0
\end{split}\end{equation*}
for $\cc$ small.

For $B=\Omega_{xx}\Omega_{yy}$, using \eqref{eqn:lambda0} and  the expansions \eqref{eqn:d_power} and \eqref{eqn:x_r_cube_inv_power}. we obtain
\begin{equation*}\begin{split}
 B&= -(3d_0)\left(5\lambda_{20}-\frac{2}{\rxo^3}\right)+O(\cc)\\
 &= -(3d_0)\left(5\lambda_{20}-2\lambda_{20}\right)+O(\cc)\\
 &=-9d_0\left(\frac{3}{2}+\frac{3}{2}d_0\right)\\
 &< 0
\end{split}\end{equation*}
for $\cc$ small.

For $D=A^2-4B$, using \eqref{eqn:lambda0} and  the expansions \eqref{eqn:d_power} and \eqref{eqn:x_r_cube_inv_power} we have
\begin{equation*}\begin{split}
 D &=\left(1-\lambda_{20}\right)^2+36d_0\lambda_{20}+O(\cc)\\
  & >0.
\end{split}\end{equation*}
for $\cc$ small.

We have proved the following result:

\begin{prop}\label{prop:x_lin_stab}
Consider the equilibria on the $x$-axis. For   $\mu\in(0,1/2]$  and for parameter $\cc$ negative and small enough, $\Omega_{zz}$ is  negative,  $A$ and $B$ are negative, and the value of the discriminant $D$ is always  positive. Consequently, two  pairs of eigenvalues are   purely imaginary, and one pair of eigenvalues are real (one positive and one negative). The linear stability is of  center $\times$ center $\times$ saddle type.
\end{prop}

\section*{Acknowledgements}\marginpar{MG: please check  acknowledgements}
This material is based upon work supported by the National Science Foundation under Grant No. DMS-1440140 while A.C. and M.G. were in residence at the Mathematical Sciences Research Institute in Berkeley, California, during the Fall 2018 semester.

This research was carried out (in part) at the Jet Propulsion Laboratory,
California Institute of Technology, under a contract with the National
Aeronautics and Space Administration and funded through the Internal
Strategic University Research Partnerships (SURP) program.

A.C. was partially supported by GNFM-INdAM and
acknowledges the MIUR Excellence Department Project awarded to the Department of Mathematics,
University of Rome Tor Vergata, CUP E83C18000100006.
M.G. and W-T.L. were partially supported by NSF grant DMS-0635607 and DMS-1814543.

We are grateful to Rodney Anderson, Edward Belbruno, Ernesto Perez-Chavela, and Pablo Rold\'an for discussions and comments.

\bibliographystyle{alpha}
\bibliography{hektor_references}

\end{document}